\documentclass[10pt,a4paper,english]{scrartcl}
\pdfoutput=1

\usepackage[T1]{fontenc}
\usepackage[utf8]{inputenc}
\usepackage{babel}

\usepackage[sc]{mathpazo}   
\usepackage[scaled=.95]{helvet}
\usepackage{courier}

\usepackage{microtype}

\KOMAoptions{DIV=last}						

\usepackage{amssymb,amsthm, amsmath,mathtools}		

\usepackage{dsfont}						

\usepackage{tikz}							
\usetikzlibrary[matrix,arrows,calc]
\usepackage{enumitem}

\usepackage{subcaption}

\usepackage{todonotes}

\usepackage{hyperref}
\hypersetup{
  colorlinks   = true,          
  urlcolor     = blue,    
  linkcolor    = blue,       
  citecolor   = purple        
}

\newcommand{\xqedhere}[2]{%
  \rlap{\hbox to#1{\hfil\llap{\ensuremath{#2}}}}}

\newcommand{\xqed}[1]{%
  \leavevmode\unskip\penalty9999 \hbox{}\nobreak\hfill
  \quad\hbox{\ensuremath{#1}}}

\let\OLDthebibliography\thebibliography
\renewcommand\thebibliography[1]{
  \OLDthebibliography{#1}
  \small
  \setlength{\parskip}{0pt}
  \setlength{\itemsep}{1.5pt plus 0.3ex}
}

 \theoremstyle{plain}
 \newtheorem{thm}{Theorem}[section]
  \theoremstyle{definition}
  \newtheorem{defn}[thm]{Definition}
  \theoremstyle{definition}
  \newtheorem{example}[thm]{Example}
  \theoremstyle{plain}
  \newtheorem{lem}[thm]{Lemma}
  \theoremstyle{plain}
  \newtheorem{cor}[thm]{Corollary}
  \theoremstyle{remark}
  \newtheorem{rem}[thm]{Remark}
  \theoremstyle{plain}
  \newtheorem{prop}[thm]{Proposition}
  \theoremstyle{plain}
  
  \theoremstyle{remark}
  
  \theoremstyle{remark}
  
  \theoremstyle{plain}
  
  \theoremstyle{definition}
  \newtheorem{constr}[thm]{Construction}
  \numberwithin{equation}{section}

\DeclareMathOperator{\Hom}{Hom}

\DeclareMathOperator{\Spec}{Spec}
\DeclareMathOperator{\Trop}{Trop}
\DeclareMathOperator{\trop}{trop}
\DeclareMathOperator{\Star}{S}

\DeclareMathOperator{\relint}{relint}

\DeclareMathOperator{\ord}{ord}
\DeclareMathOperator{\divv}{div}
\DeclareMathOperator{\Div}{Div}

\DeclareMathOperator{\Pic}{Pic}

\DeclareMathOperator{\Lin}{Span}

\DeclareMathOperator{\id}{id}

\DeclareMathOperator{\mult}{mult}

\DeclareMathOperator{\dist}{dist}

\DeclareMathOperator{\pr}{pr}

\DeclareMathOperator{\an}{an}
\DeclareMathOperator{\Mink}{M}
\DeclareMathOperator{\charak}{char}

\DeclareMathOperator{\Orb}{O}

\DeclareMathOperator{\cOrb}{V}

\DeclareMathOperator{\ClTLT}{ClCP}
\DeclareMathOperator{\TLT}{CP}
\newcommand{\Bound}{\Div_{X_0}(X)}

\newcommand{\R}{{\mathds R}}
\newcommand{\G}{{\mathds G}}
\newcommand{\Q}{{\mathds Q}}
\newcommand{\Z}{{\mathds Z}}
\newcommand{\N}{{\mathds N}}

\newcommand{\Pj}{{\mathds P}}
\newcommand{\gz}{_{\geq 0}}
\newcommand{\lz}{_{\leq 0}}
\newcommand{\ind}[1]{\operatorname{index}\left(#1\right)}
\newcommand{\varprod}[2]{#1\cup #2}

\newcommand{\tlt}{cp}
\newcommand{\tR}{\Pi}
\newcommand{\ctR}{\overline\Pi\vphantom{\Pi}}
\newcommand{\idR}{\psi_{\Pi^1}}

\renewcommand{\phi}{\varphi}

\newcommand{\Lambvee}{\Lambda\mkern-4mu^\vee}

\newcommand{\Mnulln}[1][n]{M_{0,#1}}
\newcommand{\Mnullnbar}[1][n]{\overline M_{0,#1}}
\newcommand{\Mnullnlab}{\operatorname{LSM}^\circ}
\newcommand{\Mnullnbarlab}{\operatorname{LSM}}

\newcommand{\Mnullntrop}[1][n]{M^{\trop}_{0,#1}}
\newcommand{\Mnullnbartrop}[1][n]{\overline M^{\trop}_{0,#1}}
\newcommand{\Mnullnlabtrop}{M_{0,n}^{\operatorname {lab}}}
\newcommand{\Mnullnbarlabtrop}{\overline M_{0,n}^{\operatorname{lab}}}

\newcommand{\ev}{\operatorname{ev}}

\addtokomafont{section}{\large\rmfamily}
\addtokomafont{subsection}{\normalsize\rmfamily}
\addtokomafont{title}{\rmfamily\mdseries\scshape}	

\author{\normalsize Andreas Gross}
\title{\large Intersection Theory on Tropicalizations of Toroidal Embeddings}
\date{}

\begin{document}

\maketitle
\vspace{-1.2cm}

\begin{abstract}
We show how to equip the cone complexes of toroidal embeddings with additional structure that allows to define a balancing condition for weighted subcomplexes. We then proceed to develop the foundations of an intersection theory on cone complexes including push-forwards, intersections with tropical divisors, and rational equivalence. These constructions are shown to have an algebraic interpretation: Ulirsch's tropicalizations of subvarieties of toroidal embeddings carry natural multiplicities making them tropical cycles, and the induced tropicalization map for cycles respects push-forwards, intersections with boundary divisors, and rational equivalence. As an application we prove a correspondence between the genus $0$ tropical descendant Gromov-Witten invariants introduced by Markwig and Rau and the genus $0$ logarithmic descendant Gromov-Witten invariants of toric varieties.
\end{abstract}

\section{Introduction}

The process of tropicalizing subvarieties of algebraic tori over a field with trivial valuation is at the heart of tropical geometry. Accordingly, a lot of work has been done to understand the structure of tropicalizations and the ways in which they reflect algebraic properties and constructions. The main idea of tropical geometry, namely that of transforming algebraic geometry into discrete geometry, is based on the result that the tropicalization of a subvariety $Z\subseteq \G_m^n$ is the support of a purely $\dim(Z)$-dimensional rational polyhedral fan in $\R^n$ \cite{BG84}. It can be enriched with the structure of what is known as a tropical cycle by defining multiplicities on its maximal cones \cite{Speyer-diss,ST08}. The tropicalization tells us which torus orbits are met by the closure of $Z$ in a toric compactification of $\G_m^n$ \cite{Tev07}, and, in case these intersections are proper and the compactification is smooth, the intersection multiplicities \cite{K09}. This relation between intersection theory on toric varieties and tropical geometry has been extended by the development of tropical intersection theory on $\R^n$ \cite{AR10} which incorporates the intersection rings of all normal toric compactifications of $\G_m^n$ \cite{FS97,Rau08,Katz12}. Being able to intersect tropically as well as algebraically, it makes sense to investigate in how far tropicalization respects intersections \cite{OP13,OR11} or other intersection theoretic constructions, as for example push-forwards \cite{ST08} (\cite{BPR11,OP13} in the case of non-trivial valuations).

When we want to apply tropical intersection theory to describe intersections on a non-toric variety $X$, we need to embed it into a toric variety. We then have to determine what it means for a cocycle on $X$ to be represented by tropical data on the tropicalization of the part $X_0$ of $X$ which is mapped to the big open torus. A widely accepted notion of representation has been introduced by Eric Katz \cite{Katz12}, who applied it  for $X$ equal to the moduli space of $n$-marked rational stable curves and $X_0$ the open subset of irreducible curves to relate algebraic and tropical $\psi$-classes. In \cite{CMR14} it has been used in the case where $X$ is the moduli space of relative stable maps to $\Pj^1$ and $X_0$ is the open subset corresponding to maps with irreducible domain to show that the tropical and algebraic genus $0$ relative descendant Gromov-Witten invariants of $\Pj^1$ coincide. Another approach to obtain such correspondence theorems, considered by the author in \cite{G14}, is to ignore the boundary of $X$ and consider tropicalizations of subvarieties of $X_0$ of the form $f_1^{-1} V_1 \cap\dots \cap f_n^{-1} V_n$, where the $f_i$ are maps to algebraic tori which can be extended to morphisms of algebraic tori, and the $V_i$ are generic in an appropriate sense. The main difficulty in both approaches is to find appropriate embeddings into toric varieties of the objects under consideration, and check that all maps between them can be extended to toric morphisms. On the other hand, if $X_0\subseteq X$ is a toroidal embedding, it automatically comes with a combinatorial object, namely its associated cone complex $\Sigma(X)$ \cite{KKMSD73}, and subvarieties of $X_0$ can be tropicalized to subsets of $\Sigma(X)$ in a functorial way, that is respecting toroidal morphisms \cite{Uli13}. The objective of this paper is to develop the foundations of an intersection theory on cone complexes which is compatible with tropicalization, hence making the intersection theory on $X$ accessible to tropical methods without having to find an embedding into a toric variety.

This goal breaks down into two separate but related tasks. The first one is to develop an intersection theory on  what we will call weakly embedded extended cone complexes. Our constructions and results which only need the finite part of these compactified cone complexes extend the theory of Allermann and Rau. Constructions involving compactifications of cone complexes have appeared before in the work of Shaw \cite{ShawDiss}  and Meyer \cite{HenningDiss}. In Shaw's work, she mainly considers locally matroidal tropical manifolds, which differs from our setup in that we do not consider local charts, but instead a global projection of the complex. The work of Meyer treats Kajiwara's and Payne's tropical toric varieties. These are the prototypical examples of weakly embedded extended cone complexes, yet his definitions and constructions vary significantly from ours.

The second task is to define tropicalizations of cycles on toroidal embeddings and investigate the relation between algebraic and tropical intersection theory. The underlying set of a tropicalization has already been defined in analogy to the toric case: it is the image of a suitable subset of the analytification under the tropicalization map. The precursor of the tropicalization map for a toroidal embedding $X$ is the ``$\ord$''-map $X_0\big(k\big(\!(t)\!\big)\!\big)\cap X\big(k[\mspace{-2mu}[t]\mspace{-2mu}]\big)\rightarrow\Sigma(X)$ which already appeared in \cite{KKMSD73}. This map has been greatly extended by Thuillier \cite{Thu07} who constructed a retraction of the analytification $X^\beth$ onto its skeleton, a polyhedral set which has a natural identification with the extended cone complex $\overline\Sigma(X)$. An even more general framework has been considered by Ulirsch \cite{Uli13} who constructed functorial maps $\trop_X\colon X^\beth\rightarrow \overline\Sigma(X)$ for fine and saturated logarithmic schemes $X$. It was also Ulirsch who suggested to use these maps to define set-theoretical tropicalizations and showed that in the toroidal case they have similar properties as in the toric case \cite{Uli15}. What we do in the present paper is to define multiplicities on these tropicalizations and to investigate in how far the resulting tropicalization map for cycles preserves information about the intersection theory of $X$.

The outline of the paper is as follows. In Section \ref{CCsec} we briefly recall basic facts about cone complexes and their extensions. We will introduce the central tropical objects of our theory, weakly embedded cone complexes. These are cone complexes together with a continuous piecewise integral linear map into a vector space with integral structure. We will also show how to naturally assign a weakly embedded cone complex to a strict toroidal embedding and how the boundary of the toroidal embedding is related to the boundary of the associated weakly embedded extended complex. In section \ref{ITsec} we begin the development of an intersection theory on weakly embedded extended cone complexes. We introduce Minkowski weights and tropical cycles, the analogues of cocycles and cycles on toroidal embeddings. These can be intersected with tropical Cartier divisors in two ways, one being analogous to the cup-product of a cocycle with the first Chern class of a boundary divisor on a toroidal embedding, the other to the proper intersection of a cycle with a boundary divisor. Furthermore, we define a push-forward of tropical cycles for morphisms of weakly embedded extended cone complexes. We also give a definition of rational equivalence for tropical cycles and show that intersections with divisors and push-forwards pass to the tropical Chow groups. In Section \ref{RATsec} we relate the intersection theory of a toroidal embedding $X$ to that of its weakly embedded extended cone complex $\overline\Sigma(X)$. We start by tropicalizing cocycles on $X$. The results will be Minkowski weights on $\Sigma(X)$, that is weights on the cones of $\Sigma(X)$ of appropriate dimension that satisfy the so-called balancing condition. Afterwards, we will use this construction to tropicalize cycles. We then prove that tropicalization respects push-forwards, intersections with boundary divisors, and rational equivalence. Finally, in Section \ref{Asec} we exemplify how our methods can be applied to enumerative problems in algebraic and tropical geometry. Our tropicalization procedure yields intrinsic tropicalizations of $\psi$-classes on the moduli spaces of marked rational stable curves which we show to coincide with the tropical $\psi$-classes of Mikhalkin \cite{Mik07} and Kerber and Markwig \cite{MK09}. Building on results of Ranganathan \cite{Ranga15} this leads to a correspondence between the genus $0$ tropical descendant Gromov-Witten invariants of $\R^n$ of Markwig and Rau \cite{MR08,Rau08} and the genus $0$ logarithmic descendant Gromov-Witten invariants of toric varieties. 

\paragraph{Acknowledgments} I would like to thank Andreas Gathmann for many helpful discussions and especially for numerous comments and corrections on earlier drafts. I have also benefited from several illuminating conversations with Martin Ulirsch.

\section{Cone Complexes and Toroidal Embeddings}
\label{CCsec}

The purpose of this section is to briefly recall the definitions of cone complexes and extended cone complexes. We refer to \cite[Section II.1]{KKMSD73} and \cite[Section 2]{ACP14} for further details. We will also introduce weakly embedded cone complexes, the main objects of study in the next section, and show how to obtain them from toroidal embeddings. 	

\subsection{Cone Complexes}
An \emph{(integral) cone} is a pair $(\sigma, M)$ consisting of a topological space $\sigma$ and a lattice $M$ of real-valued continuous functions on $\sigma$ such that the product map $\sigma\rightarrow\Hom(M,\mathds R)$ is a homeomorphism onto a strongly convex rational polyhedral cone. By abuse of notation we usually just write $\sigma$ for the cone $(\sigma, M)$ and refer to $M$ as $M^\sigma$. Furthermore, we will write $N^\sigma=\Hom(M,\Z)$ so that $\sigma$ is identified with a rational cone in $N^\sigma_\mathds R= N^\sigma\otimes_{\mathds Z}\mathds R$. We will also use the notation $M^\sigma_+$ for the semigroup of non-negative functions in $M$, and $N_+^\sigma$ for $\sigma\cap N^\sigma$. Giving an integral cone is equivalent to specifying a lattice $N$ and a full-dimensional strongly convex rational polyhedral cone in $N_{\mathds R}$. A subspace $\tau\subseteq\sigma$ is called a \emph{subcone} of $\sigma$ if the restrictions of the functions in $M^\sigma$ to $\tau$ make it a cone. Identifying $\sigma$ with its image in $N^\sigma_\R$, this is the case if and only if $\tau$ itself is a rational polyhedral cone in $N^\sigma_\R$.
A \emph{morphism} between two integral cones $\sigma$ and $\tau$ is a continuous map $f\colon\sigma\rightarrow\tau$ such that $m\circ f\in M^\sigma$ for all $m\in M^\tau$. The product $\sigma\times\tau$ of two cones $\sigma$ and $\tau$ in the category of cones is given by the topological space $\sigma\times\tau$, together with the functions
\begin{equation*}
m+m'\colon\sigma\times\tau\rightarrow\R,\;\; (x,y)\mapsto m(x)+m'(y)
\end{equation*} 
for $m\in M^\sigma$ and $m'\in M^\tau$. The notation already suggests that we can identify $M^\sigma\oplus M^\tau$ with $M^{\sigma\times\tau}$ via the isomorphism $(m,m')\mapsto m+m'$.
 
An \emph{(integral) cone complex} $(\Sigma,|\Sigma|)$ consists of a topological space $|\Sigma|$ and a finite set $\Sigma$ of integral cones which are closed subspaces of $|\Sigma|$, such that their union is $|\Sigma|$, every face of a cone in $\Sigma$ is in $\Sigma$ again, and the intersection of two cones in $\Sigma$ is a union of common faces. A subset $\tau\subseteq|\Sigma|$ is called a \emph{subcone} of $\Sigma$ if there exists some $\sigma\in\Sigma$ such that $\tau$ is a subcone of $\sigma$.
A \emph{morphism} $f\colon\Sigma\rightarrow\Delta$ of cone complexes is a continuous map $|\Sigma|\rightarrow\ |\Delta|$ such that for every $\sigma\in\Sigma$ there exists $\delta\in\Delta$ such that $f(\sigma)\subseteq\delta$, and the restriction $f|_\sigma\colon\sigma\rightarrow\delta$ is a morphism of cones. 

A \emph{subdivision} of a cone complex $\Sigma$ is a cone complex $\Sigma'$ such that $|\Sigma'|$ is a subspace of $|\Sigma|$ and every cone of $\Sigma'$ is a subcone of $\Sigma$. The subdivision is called \emph{proper} if $|\Sigma'|=|\Sigma|$.

The reason why the tropical intersection theory developed in \cite{AR10} fails to generalize to cone complexes is that when gluing two cones $\sigma$ and $\sigma'$ along a common face, there is no natural lattice containing both $N^\sigma$ and $N^{\sigma'}$, hence making it impossible to speak about balancing. To remedy this we make the following definition that interpolates between fans and cone complexes:

\begin{defn}
A \emph{weakly embedded cone complex} is a cone complex $\Sigma$, together with a lattice $N^\Sigma$, and a continuous map $\phi_\Sigma\colon|\Sigma|\rightarrow N^\Sigma_\R$ which is integral linear on every cone of $\Sigma$. 
\end{defn}

We recover the notion of fans by imposing the additional requirements that $\phi_\Sigma$ is injective, and the lattices spanned by $\phi_\Sigma(N_+^\sigma)$ for $\sigma\in\Sigma$ are saturated in $N^\Sigma$. In this case we also call $\Sigma$ an embedded cone complex. On the other hand, the notion of cone complexes is recovered by setting $N^\Sigma=0$.

Given a weakly embedded cone complex $\Sigma$ we write $M^\Sigma=\Hom(N^\Sigma,\Z)$ for the dual of $N^\Sigma$, and for $\sigma\in\Sigma$ we write $N^\Sigma_\sigma$ for what is usually denoted by $N^\Sigma_{\phi_\Sigma(\sigma)}$ in toric geometry, that is $N^\Sigma_\sigma= N^\Sigma\cap \Lin{\phi_\Sigma(\sigma)}$. Furthermore, we will frequently abuse notation and write $\phi_\Sigma$ for the induced morphisms $N^\sigma\rightarrow N^\Sigma$ and $N^\sigma_\R\rightarrow N^\Sigma_\R$. 

A \emph{morphism} between two weakly embedded cone complexes $\Sigma$ and $\Delta$ is comprised of a morphism $\Sigma\rightarrow \Delta$ of cone complexes and a morphism $N^\Sigma\rightarrow N^\Delta$ of lattices forming a commutative square with the weak embeddings.

Let $\Sigma$ be a cone complex, and let $\tau\in\Sigma$. For every $\sigma\in\Sigma$ containing $\tau$ let $\sigma/\tau$ be the image of $\sigma$ in $N^\sigma_\R/N^\tau_\R$. Whenever $\sigma$ and  $\sigma'$ are two cones containing $\tau$ such that $\sigma'$ is a face of $\sigma$, the cone $\sigma'/\tau$ is a naturally identified with a face of $\sigma/\tau$. Gluing the cones $\sigma/\tau$ for $\tau\preceq\sigma\in\Sigma$ along these face maps produces a new cone complex, the \emph{star $\Star_\Sigma(\tau)$} (or just $\Star(\tau)$ if $\Sigma$ is clear) of $\Sigma$ at $\tau$. If $\Sigma$ comes with a weak embedding, then the star is naturally a weakly embedded complex again. Namely, for every cone $\sigma$ of $\Sigma$ containing $\tau$ there is an induced integral linear map $\sigma/\tau\rightarrow (N^\Sigma/N^\Sigma_\tau)_\R\eqqcolon N^{\Star(\tau)}_\R$, and these maps glue to give a continuous map $\phi_{\Star(\tau)}\colon|\Star(\tau)|\rightarrow N^{\Star(\tau)}_\R$.

The product $\Sigma\times\Delta$ of two cone complexes $\Sigma$ and $\Delta$ in the category of cone complexes is the cone complex with underlying space $|\Sigma|\times|\Delta|$ and cones $\sigma\times\delta$ for $\sigma\in\Sigma$ and $\delta\in\Delta$. If both $\Sigma$ and $\Delta$ are weakly embedded, then their product in the category of weakly embedded cone complexes is the product of the cone complexes together with the weak embedding $\phi_{\Sigma\times\Delta}=\phi_\Sigma\times\phi_\Delta$.

\subsection{Extended Cone Complexes}

In the image of an integral cone $\sigma$ under its canonical embedding into $\Hom(M^\sigma,\mathds R)$  are exactly those morphisms $M^\sigma\rightarrow \R$ that are non-negative on $M^\sigma_+$. Therefore, $\sigma$ is canonically identified with the set $\Hom(M^\sigma_+,\R\gz)$ of morphisms of monoids. This identification motivates the definition of the \emph{extended cone} $\overline\sigma=\Hom(M^\sigma_+,\overline\R\gz)$ of $\sigma$, where $\overline\R\gz=\R\gz\cup\{\infty\}$, and the topology on $\overline\sigma$ is that of pointwise convergence. The cone $\sigma$ is an open dense subset of $\overline\sigma$, and $\overline\sigma$ is compact by Tychonoff's theorem. If $v\in\overline\sigma$ is an element of this compactification, the sets $\{m\in M^\sigma_+\mid \langle m,v\rangle\in\R\gz\}$ and $\{m\in M^\sigma_+\mid \langle m, v\rangle=0\}$ generate two faces of $\sigma^\vee=\R\gz M^\sigma_+\subseteq M_\R^\sigma$, and these are dual to two comparable faces of $\sigma$. In this way we obtain a stratification of $\overline\sigma$, the stratum corresponding to a pair $\tau\supseteq\tau'$ of faces of $\sigma$ being
\begin{equation*}
F^\circ_\sigma(\tau,\tau')=
\left\{v\in\overline\sigma\left\lvert
\begin{aligned}
\langle m,v\rangle\in\R\gz & \Leftrightarrow m\in(\tau')^\perp\cap M^\sigma_+,\\
\langle m,v\rangle=0	&\Leftrightarrow m\in\tau^\perp\cap M^\sigma_+
\end{aligned} 
\right.\right\}.
\end{equation*}
Denote by $F_\sigma(\tau,\tau')$ the subset of $\overline\sigma$ obtained by relaxing the second condition in the definition of $F_\sigma^\circ(\tau,\tau')$ and allowing the vanishing locus of $v$ to be possibly larger than $\tau^\perp\cap M^\sigma_+$. Since $(\tau')^\perp\cap M^\sigma$ is canonically identified with the dual lattice of $N^\sigma/N^\sigma_{\tau'}$, we can identify $F_\sigma(\tau,\tau')$ with the image $\tau/\tau'$ of $\tau$ in $(N^\sigma/N^\sigma_{\tau'})_\R$. This gives $F_\sigma(\tau,\tau')$ the structure of an integral cone.

Every morphism $f\colon\tau\rightarrow\sigma$ of cones induces a morphism $\overline f\colon\overline\tau\rightarrow\overline\sigma$ of extended cones, and if $f$ identifies $\tau$ with a face of $\sigma$, this extended morphism maps $\overline\tau$ homeomorphically onto $\coprod_{\tau'\preceq\tau}F_\sigma(\tau,\tau')=\overline{F(\tau,0)}$. Every subset of $\overline\sigma$ occurring like this is called an extended face of $\overline\sigma$. 

Given a cone complex $\Sigma$ we can glue the extensions of its cones along their extended faces according to the inclusion relation on $\Sigma$ and obtain a compactification $\overline\Sigma$ of $\Sigma$, which is called the extended cone complex associated to $\Sigma$. For a cone $\tau\in\Sigma$, the cones $\sigma/\tau=F_\sigma(\sigma,\tau)$, where $\sigma\in\Sigma$ with $\tau\subseteq\sigma$, are glued in $\overline\Sigma$ exactly as in the construction of $\Star(\tau)$, and therefore there is an identification of $\Star(\tau)$ with a locally closed subset of $\overline\Sigma$ which extends naturally to an identification of the extended cone complex $\overline{\Star(\tau)}$ with a closed subset of $\overline\Sigma$.  With these identifications, we see that $\overline\Sigma$ is stratified by the stars of $\Sigma$ at its various cones.

For every morphism $f\colon\Sigma\rightarrow\Delta$ between two cone complexes, there is an induced map $\overline f\colon\overline\Sigma\rightarrow\overline\Delta$, and we call any map arising that way a \emph{dominant} morphism of extended cone complexes. If both $\Sigma$ and $\Delta$ are weakly embedded we require $f$ to respect these embeddings. Whenever $\sigma\in\Sigma$,  and $\delta\in\Delta$ is the minimal cone of $\Delta$ containing $f(\sigma)$, there is an induced morphism $\Star_f(\sigma)\colon\Star_\Sigma(\sigma)\rightarrow \Star_\Delta(\delta)$ of cone complexes. It is easily checked that this describes $\overline f$ on the stratum $\Star_\Sigma(\sigma)$, that is that the diagram

\begin{center}
\begin{tikzpicture}[auto]
\matrix[matrix of math nodes, row sep=5ex, column sep=4em,text height=1.5ex, text depth= .25ex]{
|(Ssig)| \Star_\Sigma(\sigma) & |(Ssigbar)| \overline{\Star_\Sigma(\sigma)} & |(Sigbar)|\overline\Sigma \\
|(Sdelt)| \Star_\Delta(\delta) & |(Sdeltbar)| \overline{\Star_\Delta(\delta)} & |(Deltbar)| \overline \Delta \\
};
\begin{scope}[->,, font=\footnotesize]
\draw (Ssig)--(Ssigbar);
\draw (Ssigbar)--(Sigbar);
\draw (Sdelt)--(Sdeltbar);
\draw (Sdeltbar)--(Deltbar);

\draw (Ssig)--node{$\Star_f(\sigma)$} (Sdelt);
\draw (Ssigbar)--node{$\overline{\Star_f(\sigma)}$} (Sdeltbar);
\draw (Sigbar)--node{$\overline f$}(Deltbar);
\end{scope}
\end{tikzpicture}
\end{center}
is commutative. If $\Sigma$ and $\Delta$ are weakly embedded, and $f$ respects the weak embeddings, then so does $S_f(\sigma)$.

In general, we define a \emph{morphism of extended cone complexes} between $\overline\Sigma$ and $\overline\Delta$ to be a map $\overline\Sigma\rightarrow\overline\Delta$ which factors through a dominant morphism to $\overline{\Star_\Delta(\delta)}$ for some $\delta\in\Delta$. If $\Sigma$ and $\Delta$ are weakly embedded, we additionally require the dominant morphism to respect the weak embeddings.

\subsection{Toroidal Embeddings}

Before we start with the development of an intersection theory on weakly embedded cone complexes, we want to point out how to obtain them from toroidal embeddings as this will be our primary source of motivation and intuition.
A \emph{toroidal embedding} is pair $(X_0,X)$ consisting of a normal variety $X$ and a dense open subset $X_0\subseteq X$ such that the open immersion $X_0\rightarrow X$ formally locally looks like the inclusion of an algebraic torus $T$ into a $T$-toric variety. More precisely, this means that for every closed point $x\in X$ there exists an affine toric variety $Z$, a closed point $z\in Z$, and an isomorphism $\widehat{\mathcal O}_{X,x} \cong \widehat{\mathcal O}_{Z,z}$ over the ground field $k$ which identifies the ideal of $X\setminus X_0$ with that of $Z\setminus Z_0$, where $Z_0$ denotes the open orbit of $Z$.
A toroidal embedding is called strict, if all components of $X\setminus X_0$ are normal.
In this paper, we will only consider strict toroidal embeddings and therefore omit the "strict".

Every toroidal embedding $X$ has a canonical stratification. If $E_1,\dots, E_n$ are the components of $X\setminus X_0$, then the strata are given by the connected components of the sets
\begin{equation*}
\bigcap_{i\in I} E_i \setminus \bigcup_{j\notin I} E_j,
\end{equation*}
where $I$ is a subset of $\{1,\dots,n\}$.

The combinatorial open subset $X(Y)$ of $X$ associated to a stratum $Y$ is the union of all strata of $X$ containing $Y$ in their closure. This defines an open subset of $X$, as it is constructible and closed under generalizations. 
Furthermore, the stratum $Y$ defines the following lattices, semigroups, and cones:
\begin{align*}
M^Y&=\{\text{Cartier divisors on }X(Y) \text{ supported on } X(Y)\setminus X_0\}\\
N^Y&=\Hom(M^Y,\Z)\\
M^Y_+&=\{\text{Effective Cartier divisors in }M^Y\}\\
N^Y_{\mathds R}\supseteq\sigma^Y &= (\R\gz M^Y_+)^\vee.
\end{align*}
If $X$ is unclear, we write $M^Y(X)$, $N^Y(X)$, and $\sigma^Y_X$.
Whenever a stratum $Y$ is contained in the closure of a stratum $Y'$
the morphism $N^{Y'}\rightarrow N^Y$ induced by the restriction of divisors maps injectively onto a saturated sublattice of $N^Y$ and identifies $\sigma^{Y'}$ with a face of $\sigma^Y$. Gluing along these identifications produces the cone complex $\Sigma(X)$. We refer to \cite[Section II.1]{KKMSD73} for details. In accordance with the toric case we will write $\Orb(\sigma)$ for the stratum of $X$ corresponding to a cone $\sigma\in\Sigma(X)$. Its closure will be denoted by $\cOrb(\sigma)$, and we will abbreviate $X(\Orb(\sigma))$ to $X(\sigma)$.

Define $M^X=\Gamma(X_0,\mathcal O_X^*)/k^*$, and let $N^X=\Hom(M^X,\mathds Z)$ be its dual. For every stratum $Y$ of $X$  we have a morphism
\begin{equation*}
	M^X\rightarrow M^Y,\;\; f\mapsto \divv(f)|_{X(Y)},
\end{equation*}
which induces an integral linear map $\sigma^Y\rightarrow N^X_\R$. Obviously, these maps glue to give a continuous function $\phi_X\colon|\Sigma(X)|\rightarrow N^X_\R$ which is integral linear on the cones of $\Sigma(X)$. In other words, we obtain a weakly embedded cone complex naturally associated to $X$, which we again denote by $\Sigma(X)$.

\begin{example}
\label{ITexample: Toric variety and P2 without two hyperplanes}
\leavevmode
\vspace{-1ex}
\begin{enumerate}[label=\alph*)]
\item Let $\Sigma$ be a fan in $N_\R$ for some lattice $N$, and let $X$ be the associated normal toric variety. Let $M=\Hom(N,\Z)$ be the dual of $N$ and $T=\Spec k[M]$ the associated algebraic torus. By definition, $T\subseteq X$ is a toroidal embedding. The components of the boundary $X\setminus T$ are the $T$-invariant divisors $D_\rho$ corresponding to the rays $\rho\in\Sigma_{(1)}$, and the strata of $X$ are the $T$-orbits $O(\sigma)$ corresponding to the cones $\sigma\in\Sigma$. The combinatorial open subsets of $X$ are precisely its $T$-invariant affine opens. For every $\tau\in\Sigma$ the isomorphism 
\begin{equation*}
M/(M\cap \tau^\perp)\rightarrow M^{O(\tau)},\;\; [m]\mapsto \divv(\chi^m),
\end{equation*}
where $\chi^m$ denotes the character associated to $m$, induces identifications of $N^{O(\tau)}$ with $N_\tau$ and $\sigma^{O(\tau)}$ with $\tau$. After identifying $M$ and $M^X$ via the isomorphism
\begin{equation*}
M\rightarrow M^X=\Gamma(T,\mathcal O_X^*)/k^*,\;\; m\mapsto \chi^m
\end{equation*}
we see that the image of $\sigma^{O(\tau)}$ in $N_\R$ under $\phi_X$ is precisely $\tau$. We conclude that $\Sigma(X)$ is an embedded cone complex, naturally isomorphic to $\Sigma$.

\item For a non-toric example consider $X=\Pj^2$ with open part $X_0= X\setminus H_1\cup H_2$, where $H_i=\{x_i=0\}$. Since $X$ is smooth, $\Sigma(X)$ is naturally identified with the orthant $(\R\gz)^2$, whose rays $\R\gz e_1$ and $\R\gz e_2$ correspond to the divisors $H_1$ and $H_2$. The lattice $M^X$ is generated by $\frac{x_1}{x_2}$, and using that generator to identify $M^X$ with $\Z$ the weak embedding $\phi_X$ sends $e_1$ to $1$ and $e_2$ to $-1$, as depicted in Figure \ref{ITfig:Weakly embedded cone complex}.
\end{enumerate}
\end{example}

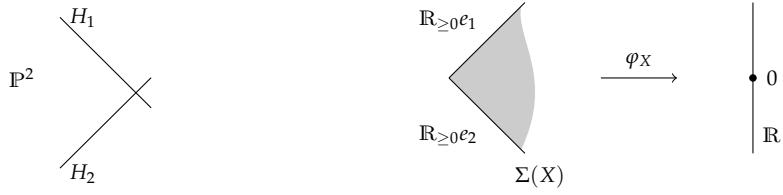
\begin{figure}

\begin{minipage}{.32\textwidth}
\centering
\begin{tikzpicture}[ font=\footnotesize, auto, near end]
\draw (.2,-.2)--node[above,near end, inner sep=1ex]{$H_1$}(-1,1);
\draw (.2,.2)--node[below,inner sep=1.5ex]{$H_2$}(-1,-1);
\node at (-1.5,0.2) {$\Pj^2$};
\end{tikzpicture}
\end{minipage}
\hfill
\begin{minipage}{.63\textwidth}
\centering
\begin{tikzpicture}[auto, font=\footnotesize]
\fill[fill=black!20] (0,0) -- ($.95*(1,1)$)	.. controls (0.87,0.53) and (1.4,0) .. ($.95*(1,-1)$) -- cycle;
\draw (0,0) -- node{$\R\gz e_1$} (1,1);
\draw (0,0) -- node[swap]{$\R\gz e_2$} (1,-1);
\begin{scope}[xshift=0cm]
\draw[->] (2,0) --node{$\phi_X$} (3,0);
\draw (4,-1) --node[very near start,swap]{$\R$} (4,1);
\node at (4,0) [shape=circle,fill=black, label=east:$0$,inner sep =1pt]{} ;
\node at (1.2,-1.3) {$\Sigma(X)$};
\end{scope}
\end{tikzpicture}
\end{minipage}
\caption{The toroidal embedding $\Pj^2\setminus (H_1\cup H_2)\subseteq \Pj^2$ and its weakly embedded cone complex}
\label{ITfig:Weakly embedded cone complex}
\end{figure}

It follows directly from the definitions that whenever $Y$ is a stratum of a toroidal embedding $X$, the embedding $Y\subseteq \overline Y$ is toroidal again. In the following lemma we compare the weakly embedded cone complex of $\overline Y$ with the star of $\Sigma(X)$ at $\sigma^Y$.

\begin{lem}
\label{ITlem: complex of stratum}
Let $X_0\subseteq X$ be a toroidal embedding, and let $Y$ and $Y'$ be two strata of $X$ such that $Y'\subseteq\overline Y$. Then the sublattice $M^{Y'}(X)\cap (\sigma_X^Y)^\perp$ consists precisely of the divisors in $M^{Y'}(X)$ whose support does not contain $Y$. The natural restriction map 
\begin{equation*}
M^{Y'}(X)\cap (\sigma_X^Y)^\perp \rightarrow M^{Y'}(\overline Y)
\end{equation*}
is an isomorphism identifying $M_+^{Y'}(X)\cap (\sigma_X^Y)^\perp$ with $M_+^{Y'}(\overline Y)$. Therefore, it induces an isomorphism $\sigma_{\overline Y}^{Y'}\xrightarrow{\simeq}\sigma_X^{Y'}/\sigma_X^Y$ of cones. These isomorphisms glue and give an identification of $\Sigma(\overline Y)$ with $\Star_{\Sigma(X)}(\sigma_X^Y)$. Furthermore, there is a natural morphism $N^{\Sigma(\overline Y)}\rightarrow N^{\Star_{\Sigma(X)}(\sigma_X^Y)}$ respecting the weak embeddings.
\end{lem}

\begin{proof}
Let $D_1,\dots,D_s$ be the boundary divisors containing $Y'$, and assume that they are labeled such that $D_1,\dots, D_r$ are the ones containing $Y$. Let $u_i$ be the primitive generator of the ray $\sigma^{D_i}$. For a Cartier divisor $D=\sum_i a_i D_i\in M^{Y'}(X)$ we have $\langle D,u_i\rangle=a_i$ by \cite[p. 63]{KKMSD73}. It follows that $M^{Y'}(X)\cap (\sigma^Y)^\perp$ contains exactly those divisors of $M^{Y'}(X)$ with $a_1=\dots= a_r=0$, which are precisely those whose support does not contain $Y$. These divisors can be restricted to the combinatorial open subset $\overline Y(Y')$ of $\overline Y$, yielding divisors in $M^{Y'}(\overline Y)$. To show that the restriction map is an isomorphism we may reduce to the toric case by choosing a local toric model at a closed point of $Y$ and using \cite[II, §1,  Lemma 1]{KKMSD73}. So assume $X=U_\sigma$ is an  affine toric variety defined by a cone $\sigma\subset N_\R$ in a lattice $N$, the stratum $Y'=O(\sigma)$ is the closed orbit, and $Y=O(\tau)$ for a face $\tau\prec\sigma$. Further reducing to the case in which $X$ has no torus factors we may assume that $\sigma$ is full-dimensional. Let $M$ be the dual of $N$. The isomorphism 
\begin{equation*}
M\rightarrow M^{Y'}(X),\;\; m\mapsto \divv(\chi^m),
\end{equation*}
where $\chi^m$ is the character associated to $m$, induces a commutative diagram
\begin{center}
\begin{tikzpicture}
\matrix[matrix of math nodes, row sep=4ex,text height=1.5ex, text depth= .25ex]{
|(M)| M 		&[1em]	|(Mtau)| M\cap \tau^\perp &[3em] \\
|(MY')| M^{Y'}(X)	&	|(MYsig)| M^{Y'}(X)\cap (\sigma^Y)^\perp 	&	|(MY'Y)| M^{Y'}(\overline Y) \\
};
\begin{scope}[->, font=\footnotesize]
\path (M) --node{$\supseteq$} (Mtau);
\draw (Mtau.east) -- (MY'Y);
\path (MY') --node{$\supseteq$} (MYsig);
\draw (MYsig) -- (MY'Y);

\draw (M) --node[left]{$\cong$} (MY');
\draw (Mtau) --node[left]{$\cong$} (MYsig);

\end{scope}	 
\end{tikzpicture}
\end{center}
It is well known that the character lattice of the open torus embedded in the toric variety $\overline Y$ can be naturally identified with $M\cap\tau^\perp$. With this identification, the upper right morphism in the diagram sends a character in $M\cap\tau^\perp$ to its associated principal divisor. Thus, it is an isomorphism which implies that the restriction map $M^{Y'}(X)\cap(\sigma^Y)^\perp\rightarrow M^{Y'}(\overline Y)$ is an isomorphism as well. Both $M_+^{Y'}(X)\cap (\sigma_X^Y)^\perp$ and $M_+^{Y'}(\overline Y)$ correspond to $M\cap\tau^\perp\cap\sigma^\vee$, hence they get identified by the restriction map. Dualization induces an isomorphism of the dual cone $\sigma_{\overline Y}^{Y'}$ of $M_+^{Y'}(\overline Y)$ and the dual cone $\sigma_X^{Y'}/\sigma_X^Y$ of $M_+^{Y}(X)\cap(\sigma_X^Y)^\perp$. If $Y''$ is a third stratum of $X$ such that $Y''\subseteq \overline {Y'}$, then the diagram
\begin{center}
\begin{tikzpicture}
\matrix[matrix of math nodes, row sep=4ex, column sep= 3em, text height= 1.5ex, text depth=.25ex]{
|(MY'')| M^{Y''}(X)\cap (\sigma_X^Y)^\perp &
|(MY''Y)| M^{Y''}(\overline Y) \\
|(MY')| M^{Y'}(X)\cap (\sigma_X^Y)^\perp &
|(MY'Y)| M^{Y'}(\overline Y)\\
};

\begin{scope}[->]
\draw (MY'')-- (MY''Y);
\draw (MY') -- (MY'Y);
\draw (MY'') -- (MY');
\draw (MY''Y) -- (MY'Y);
\end{scope}
\end{tikzpicture}
\end{center}
is commutative because all maps involved are restrictions to open or closed subschemes. It follows that the isomorphisms of cones glue to an isomorphism $\Sigma(\overline Y)\cong \Star_{\Sigma(X)}(\sigma_X^Y)$. To see that this isomorphism respects the weak embeddings, note that $M^{\Star_{\Sigma(X)}(\sigma_X^Y)}$ is equal to the sublattice $(N^X_{\sigma^Y_X})^\perp$ of $M^X$ by definition. It consists exactly of those rational functions that are invertible on $X(Y)$. Hence, they can be restricted to $\overline Y$, giving a morphism 
\begin{equation*}
M^{\Star_{\Sigma(X)}(\sigma_X^Y)}\rightarrow M^{\overline Y}.
\end{equation*}
For any rational function $f$ on $X$ that is invertible on $Y$ we have $\divv(f|_{\overline Y})=\divv(f)|_{\overline Y}$. It follows directly from this equality that the identification $\Sigma(\overline Y)\rightarrow \Star_{\Sigma(X)}(\sigma_X^Y)$ is in fact a morphism of weakly embedded cone complexes.
\end{proof}

A \emph{dominant toroidal} morphism between two toroidal embeddings $X$ and $Y$ is a dominant morphism $X\rightarrow Y$ of varieties which can be described by toric morphisms in local toric models (see \cite{AK00} for details). A dominant toroidal morphism $f\colon X\rightarrow Y$ induces a dominant morphism $\Trop(f)\colon\overline\Sigma(X)\rightarrow \overline\Sigma(Y)$ of extended cone complexes. The restrictions of $\Trop(f)$ to the cones of $\Sigma(X)$ are dual to pulling back Cartier divisors. From this we easily see that $\Trop(f)$ can be considered as a morphism of weakly embedded extended cone complexes by adding to it the data of the linear map $N^X\rightarrow N^Y$ dual to the pullback $\Gamma(Y_0,\mathcal O_Y^*)\rightarrow \Gamma(X_0,\mathcal O_X^*)$. We call a morphism $f\colon X\rightarrow Y$  \emph{toroidal} if it factors as $X\xrightarrow{f\smash{'}} \cOrb(\sigma)\xrightarrow{i} Y$, where $f'$ is dominant toroidal, and $\sigma\in\Sigma(Y)$.  By Lemma \ref{ITlem: complex of stratum} the closed immersion $i$ induces a morphism $\Trop(i)\colon\overline\Sigma(\cOrb(\sigma))\rightarrow\overline\Sigma(Y)$ of weakly embedded extended cone complexes, namely the composite of the canonical  morphism $\overline\Sigma(\cOrb(\sigma))\rightarrow \overline{\Star_{\Sigma(Y)}(\sigma)}$ and the inclusion of $\overline{\Star_{\Sigma(Y)}(\sigma)}$ in $\overline\Sigma(Y)$. Thus we can define $\Trop(f)\colon\overline\Sigma(X)\rightarrow\overline\Sigma(Y)$ as the composite $\Trop(i)\circ\Trop(f')$.

A special class of toroidal morphisms is given by \emph{toroidal modifications}. They are the analogues of the toric morphisms resulting from refinements of fans in toric geometry. For every  subdivision $\Sigma'$ of the cone complex $\Sigma$ of a toroidal embedding $X$ there is a unique toroidal modification $X\times_\Sigma\Sigma'\rightarrow X$ whose tropicalization is the subdivision $\Sigma'\rightarrow \Sigma$ \cite{KKMSD73}. This modification maps $(X\times_\Sigma\Sigma')_0$ isomorphically onto $X_0$. Modifications are compatible with toroidal morphisms in the sense that if $f\colon X\rightarrow Y$ is dominant toroidal, and $\Sigma'$ and $\Delta'$ are subdivisions of $\Sigma(X)$ and $\Sigma(Y)$, respectively, such that $\Trop(f)$ induces a morphisms $\Sigma'\rightarrow \Delta'$, then $f$ lifts to a toroidal morphism $f'\colon X\times_{\Sigma(X)}\Sigma'\rightarrow Y\times_{\Sigma(Y)}\Delta$ \cite[Lemma 1.11]{AK00}.

\begin{rem}
Our usage of the term toroidal morphism is non-standard. Non-dominant toroidal morphisms in our sense are called subtoroidal in \cite{ACP14}. The notation for toroidal modifications is due to Kazuya Kato \cite{Kato94} and is not as abusive as it may seem \cite[Prop. 9.6.14]{GR14}.
\end{rem}

\section{Intersection Theory on Weakly Embedded Cone Complexes}
\label{ITsec}

In what follows we will develop the foundations of a tropical intersection theory on weakly embedded cone complexes. Our constructions are motivated by the relation of algebraic and tropical intersection theory as well as by the well-known constructions for the embedded case studied in \cite{AR10}. In fact, those of our constructions that work primarily in the finite part of the cone complex are natural generalizations of the corresponding constructions for embedded complexes. 
Intersection theoretical constructions involving boundary components at infinity have been studied in the setup of tropical manifolds \cite{Mik06,ShawDiss} and for Kajiwara's and Payne's tropical toric varieties \cite{HenningDiss}. The latter is closer to our setup, yet definitions and proofs vary significantly from ours.

\subsection{Minkowski Weights, Tropical Cycles, and Tropical Divisors}

For the definitions of Minkowski weights and tropical cycles we need the notion of lattice normal vectors. Let $\tau$ be a codimension $1$ face of a cone $\sigma$ of a weakly embedded cone complex 
$\Sigma$.  We denote by $u_{\sigma/\tau}$ the image under the  morphism
\begin{equation*}
N^{\sigma}/N^\tau\rightarrow N^\Sigma/N^\Sigma_\tau
\end{equation*}
induced by the weak embedding of the generator of  $N^{\sigma}/N^\tau$ which is contained in the image of $\sigma\cap N^{\sigma}$ under the  projection $N^{\sigma}\rightarrow N^{\sigma}/N^\tau$, and call it the \emph{lattice normal vector of $\sigma$ relative to $\tau$}. Note that lattice normal vectors may be equal to $0$.

\begin{defn}
Let $\Sigma$ be a weakly embedded cone complex. A \emph{$k$-dimensional Minkowski weight} on $\Sigma$ is a map $c\colon\Sigma_{(k)}\rightarrow \mathds Z$ from the $k$-dimensional cones of $\Sigma$ to the integers such that it satisfies the \emph{balancing condition} around every $(k-1)$-dimensional cone $\tau\in \Sigma$: if $\sigma_1,\dots,\sigma_n$ are the $k$-dimensional cones containing $\tau$, then 
\begin{equation*}
\sum_{i=1}^n c(\sigma_i)u_{\sigma_i/\tau} = 0 \qquad\text{in } N^\Sigma/N^\Sigma_\tau.
\end{equation*} 
The $k$-dimensional Minkowski weights naturally form an abelian group, which we denote by $\Mink_k(\Sigma)$.

We define the group of \emph{tropical $k$-cycles} on $\Sigma$ by $Z_k(\Sigma)=\varinjlim \Mink_k(\Sigma')$, where $\Sigma'$ runs over all proper subdivisions of $\Sigma$. If $c$ is a $k$-dimensional Minkowski weight on a proper subdivision of $\Sigma$, we denote by $[c]$ its image in $Z_k(\Sigma)$. The group of tropical $k$-cycles on the extended complex $\overline\Sigma$ is defined by
$Z_k(\overline\Sigma)=\bigoplus_{\sigma\in\Sigma} Z_k(\Star(\sigma))$. We will write $\Mink_*(\Sigma)=\bigoplus_k\Mink_k(\Sigma)$ for the graded group of Minkowski weights, and similarly $Z_*(\Sigma)$ and $Z_*(\overline\Sigma)$ for the graded groups of tropical cycles on $\Sigma$ and $\overline\Sigma$, respectively.

The support $|A|$ of a cycle $A=[c]\in Z_k(\Sigma)$ is the union of all $k$-dimensional cones on which $c$ has non-zero weight. This is easily seen to be independent of the choice of $c$.
\end{defn}

\begin{rem}
In the definition of cycles we implicitly used that a proper subdivision $\Sigma'$ of a weakly embedded cone complex $\Sigma$ induces a morphism $\Mink_k(\Sigma)\rightarrow \Mink_k(\Sigma')$. The definition of this morphism is clear, yet a small argument is needed to show that balancing is preserved. But this can be seen similarly as in the embedded case \cite[Lemma 2.11]{AR10}.
\end{rem}

\begin{rem}
\label{ITrem:Analogy between tropical and algebraic IT}
Minkowski weights are meant as the analogues of cocycles on toroidal embeddings, whereas tropical cycles are the analogues of cycles. To see the analogy consider the toric case. There, the cohomology group of a normal toric variety $X$ corresponding to a fan $\Sigma$ is canonically isomorphic to the group of Minkowski weights on $\Sigma$ \cite{FS97}. On the other hand, for a subvariety $Z$ of $X$ this isomorphism cannot be used to assign a Minkowski weight to the cycle $[Z]$ unless $X$ is smooth. In general, we first have to take a proper transform of $Z$ to a smooth toric modification of $X$, and therefore only obtain a Minkowski weight up to refinements of fans, i.e.\ a tropical cycle. That this tropical cycle is well-defined has been shown e.g.\ in \cite{ST08,K09}.
\end{rem}

\begin{example}
Consider the weakly embedded cone complex from Example \ref{ITexample: Toric variety and P2 without two hyperplanes}b) and let us denote it by $\Sigma$. Let $\sigma$ be its maximal cone, and $\rho_1=\R\gz e_1$ and $\rho_2=\R\gz e_2$ its two rays. The two lattice normal vectors $u_{\sigma/\rho_1}$ and $u_{\sigma/\rho_2}$ are equal to $0$. Therefore, the balancing condition for $2$-dimensional Minkowski weights is trivial and we have $\Mink_2(\Sigma)=\Z$. In dimension $1$, we have $u_{\rho_1/0}= 1$ and $u_{\rho_2/0}=-1$. Hence, every balanced $1$-dimensional weight must have equal weights on the two rays, which implies $\Mink_1(\Sigma)=\Z$. In dimension $0$ there is, of course, no balancing to check and we have $\Mink_0(\Sigma)=\Z$ as well.
\end{example}

Cocycles on a toroidal embedding can be pulled back to closures of strata in its boundary. This works for Minkowski weights as well:

\begin{constr}[Pullbacks of Minkowski weights]
\label{ITconstr:pulling back Minkowski weights}
Let $\Sigma$ be a weakly embedded cone complex, and let $\gamma$ be a cone of $\Sigma$. Whenever we are given an inclusion $\tau\preceq \sigma$ of two cones of $\Sigma$ which contain $\gamma$, we obtain an inclusion $\tau/\gamma\preceq\sigma/\gamma$ of cones  in $\Star(\gamma)$ of the same codimension. By construction, the restriction of the weak embedding $\phi_{\Star(\gamma)}$ to $\sigma/\gamma$ is equal to the map $\sigma/\gamma\rightarrow (N^\Sigma/N^\Sigma_{\gamma})_\R=N^{\Star(\gamma)}_\R$ induced by $\phi_\Sigma$. It follows that there is a natural isomorphism $N^\Sigma/N^\Sigma_\tau\xrightarrow{\cong} N^{\Star(\gamma)}/N^{\Star(\gamma)}_{\tau/\gamma}$. Since $N^{\sigma/\gamma}=N^\sigma/N^\gamma$ by construction, we also have a canonical isomorphism $N^\sigma/N^\tau\xrightarrow{\cong} N^{\sigma/\gamma}/N^{\tau/\gamma}$. These isomorphisms fit into a commutative diagram
\begin{center}
\begin{tikzpicture}[auto]
\matrix[matrix of math nodes, row sep= 4ex, column sep= 3em, text height=1.5ex, text depth= .25ex]{
|(sig)| N^\sigma/N^\tau 	&	|(Sig)| N^\Sigma/N^\Sigma_\tau \\
|(sigmodgamma)| N^{\sigma/\gamma}/N^{\tau/\gamma} & |(Sgamma)| N^{\Star(\gamma)}/N^{\Star(\gamma)}_{\tau/\gamma} .\\
};
\begin{scope}[->,font=\footnotesize]
\draw (sig) -- (Sig);
\draw (sigmodgamma) -- (Sgamma);
\draw (sig) --node{$\cong$} (sigmodgamma);
\draw (Sig) --node{$\cong$} (Sgamma);
\end{scope}
\end{tikzpicture}
\end{center}
If $\tau$ has codimension $1$ in $\sigma$, we see that the image of the lattice normal vector $u_{\sigma/\tau}$ in $N^{\Star(\gamma)}/N^{\Star(\gamma)}_{\tau/\gamma}$ is the lattice normal vector $u_{(\sigma/\gamma)/(\tau/\gamma)}$. Now suppose we are given a Minkowski weight $c\in \Mink_k(\Sigma)$ with $k\geq\dim\gamma$. It follows immediately from the considerations above  that the induced weight on $\Star(\gamma)$ assigning $c(\sigma)$ to $\sigma/\gamma$ for $\sigma\in\Sigma_{(k)}$ with $\gamma\preceq\sigma$ satisfies the balancing condition. Hence it defines a Minkowski weight in $\Mink_{k-\dim\gamma}(\Star(\gamma))$. If $i\colon\Star(\gamma)\rightarrow \overline\Sigma$ is the inclusion map, we denote it by $i^*c$. In case $k<\dim\gamma$ we set $i^*c=0$. \xqed{\lozenge}
\end{constr}

\begin{defn}
Let $\Sigma$ be a weakly embedded cone complex. 
A Cartier divisor on $\Sigma$ is a continuous function $\psi\colon|\Sigma|\rightarrow\R$ which is integral linear on all cones of $\Sigma$. We denote the group of Cartier divisors by $\Div(\Sigma)$. A Cartier divisor $\psi$ is said to be combinatorially principal (\tlt{} for short) if $\psi$ is induced by a function in $M^\Sigma$ on each cone. The subgroup of $\Div(\Sigma)$ consisting of \tlt{}-divisors is denoted by $\TLT(\Sigma)$. Two divisors are called linearly equivalent if their difference is induced by a function in $M^\Sigma$. We denote by $\ClTLT(\Sigma)$ the group of \tlt{}-divisors modulo linear equivalence.
\end{defn}

\begin{rem}
Note that our usage of the term Cartier divisors is non-standard. Usually, piecewise integral linear functions on $\R^n$, or more generally embedded cone complexes, are called rational functions because they arise naturally as ``quotients'', that is differences, of tropical ``polynomials'', that is functions of the form $\R^n\rightarrow\R ,\,x\mapsto \min\{\langle m ,x\rangle\mid m\in\Delta\}$ for some finite set $\Delta\subseteq \Z^n$. As we lack the embedding, we chose a different analogy: Assuming that the cone complex $\Sigma$ comes from a toroidal embedding $X$, integer linear functions on a cone $\sigma$ correspond to divisors on $X(\sigma)$ supported on $X(\sigma)\setminus X_0$. Since the cones in $\Sigma$ and the combinatorial open subsets of $X$ are glued accordingly, this induces a correspondence between $\Div(\Sigma)$ and the group of Cartier divisors on $X$ supported away from $X_0$. Unlike in the toric case, a Cartier divisor on a combinatorial open subset $X(\sigma)$ supported away from $X_0$ does not need to be principal. In case it is, it is defined by a rational function in $\Gamma(X_0,\mathcal O_X^*)$. This means that the associated linear function on $\sigma$ is the pullback of a function in $M^X$, explaining the terminology combinatorially principal. Finally, linear equivalence is defined precisely in such a way that two divisors are linearly equivalent if and only if their associated divisors on $X$ are.
\end{rem}

The functorial behavior of tropical Cartier divisors is as one would expect from algebraic geometry. They can be pulled back along dominant morphisms of weakly embedded extended cone complexes, and along arbitrary morphisms if passing to linear equivalence. In the latter case, however, we need to restrict ourselves to \tlt{}-divisors.

\begin{constr}[Pullbacks of Cartier divisors]
\label{ITconstr:Pullbacks of Cartier divisors}
Let $f\colon\overline\Sigma\rightarrow\overline\Delta$ be a morphism of weakly embedded extended cone complexes, and let $\psi$ be a divisor on $\Delta$. If $f$ is dominant, then it follows directly from the definitions that $f^*\psi=\psi\circ f$ is a divisor on $\Sigma$.  We call it the \emph{pull-back} of $\psi$ along $f$. Moreover, if $\psi$ is combinatorially principal, then so is $f^*\psi$.  The map $f^*\colon\TLT(\Delta)\rightarrow\TLT(\Sigma)$ clearly is a morphism of abelian groups, and  it preserves linear equivalence. Therefore, it induces a morphism $\ClTLT(\Delta)\rightarrow\ClTLT(\Sigma)$ which we again denote by $f^*$. If $f$ is arbitrary, it factors uniquely as $\overline\Sigma\xrightarrow{f'}\overline{\Star_\Delta(\delta)}\xrightarrow{i} \overline\Delta$, where $\delta\in\Delta$, $f'$ is dominant, and $i$ is the inclusion map. Hence, to define a pullback $f^*:\ClTLT(\Delta)\rightarrow \ClTLT(\Sigma)$ it suffices to construct pullbacks of divisor classes on $\Delta$ to $\Star_\Delta(\delta)$. So let $\psi$ be a \tlt{}-divisor on $\Delta$. Choose $\psi_\delta\in M^\Delta$ such that $\psi|_\delta=\psi_\delta\circ\phi_\Delta|_\delta$, and denote $\psi'=\psi-\psi_\delta\circ\phi_\Delta$. Then $\psi'|_\delta=0$ and thus for every $\sigma\in\Delta$ with $\delta\preceq\sigma$ there is an induced  integral linear map $\smash{\overline\psi}'_\sigma\colon\sigma/\delta\rightarrow\R$. These maps patch together to a \tlt{}-divisor $\overline\psi'$ on $\Star_\Delta(\delta)$. Clearly, $\smash{\overline\psi}'$ depends on the choice of $\psi_\delta$, but different choices produce linearly equivalent divisors. Thus, we obtain a well-defined element $i^*\psi\in\ClTLT(\Star_\Delta(\delta))$. By construction, divisors linearly equivalent to $0$ are in the kernel of $i^*$, so there is an induced morphism $\ClTLT(\Delta)\rightarrow\ClTLT(\Star_\Delta(\delta))$ which we again denote by $i^*$. Finally, we define the pullback along $f$ by $f^*\coloneqq(f')^*\circ i^*$. 
\xqed{\lozenge}
\end{constr}

\subsection{Push-forwards}

Now we want to construct push-forwards of tropical cycles. To do so we need to know that cone complexes allow sufficiently fine proper subdivisions with respect to given subcones.

\begin{lem}
\label{ITlem:fine proper subdivisions exist}
Let $\Sigma$ be a cone complex and $\mathcal A$ a finite set of subcones of $\Sigma$. Then there exists a proper subdivision $\Sigma'$ of $\Sigma$ such that every cone in $\mathcal A$ is a union of cones in $\Sigma'$.
\end{lem}

\begin{proof}
If $\Sigma$ is a fan in $N_\R$ for some lattice $N$, it is well-known how to obtain a suitable subdivision. Namely, we choose linear inequalities for every $\tau\in \mathcal A$ and intersect the cones in $\Sigma$ with the half-spaces defined by these inequalities. In the general case we may assume that $\Sigma$ is strictly simplicial by choosing a proper strictly simplicial subdivision $\Sigma'$ and intersecting all cones in $\mathcal A$ with the cones in $\Sigma'$. Let $\phi\colon|\Sigma|\rightarrow\R^{\Sigma_{(1)}}$ be the weak embedding sending the primitive generator of a ray $\rho\in\Sigma_{(1)}$ to the generator $e_\rho$ of $\R^{\Sigma_{(1)}}$ corresponding to $\rho$. Let $\Delta$ be a complete fan in $\R^{\Sigma_{(1)}}$ such that $\phi(\sigma)$ is a union of cones of $\Delta$ for every $\sigma\in\Sigma\cup \mathcal A$. We define $\Sigma'$ as the proper subdivision of $\Sigma$ consisting of the cones $\phi^{-1}\delta\cap\sigma$ for $\delta\in\Delta$ and $\sigma\in\Sigma$ and claim that it satisfies the desired property. Let $\tau\in \mathcal A$, and let $\sigma\in\Sigma$ be a cone containing it. Then by construction there exist cones $\delta_1,\dots,\delta_k$ of $\Delta$ such that $\phi(\tau)=\bigcup_i\delta_i$. Each $\delta_i$ is in the image of $\sigma$, hence $\phi(\phi^{-1}\delta_i\cap\sigma)=\delta_i$. With this it follows that $\phi(\tau)=\phi\left(\bigcup_i(\phi^{-1}\delta_i\cap\sigma)\right)$, and hence that $\tau$ is a union of cones in $\Sigma'$ since $\phi$ is injective on $\sigma$ by construction.
\end{proof}

\begin{constr}[Push-forwards]
\label{ITconstr:push-forwards}
 Let $f\colon\overline\Sigma\rightarrow\overline\Delta$ be a morphism of weakly embedded extended cone complexes, and let $A\in Z_k(\overline\Sigma)$ be a tropical cycle. We construct a push-forward $f_*A\in Z_k(\overline\Delta)$ similarly as in the embedded case \cite{GKM09}. First suppose that $f$ is dominant, and that  $A$ is represented by a Minkowski weight $c$ on a proper subdivision $\Sigma'$ of $\Sigma$. We further reduce to the case where the images of cones in $\Sigma'$ are cones in $\Delta$: for every $\sigma'\in\Sigma'$ the image $f(\sigma')$ is a subcone of $\Delta$. By Lemma \ref{ITlem:fine proper subdivisions exist} we can find a proper subdivision $\Delta'$ of $\Delta$ such that each of these subcones is union of cones of $\Delta'$. The cones $\sigma'\cap f^{-1}\delta'$ for $\sigma'\in\Sigma'$ and $\delta'\in\Delta'$ form a proper subdivision $\Sigma''$ of $\Sigma'$, and standard arguments show that $f(\sigma'')\in\Delta'$ for all $\sigma''\in\Sigma''$. As $c$ induces a Minkowski weight on $\Sigma''$, and Minkowski weights on $\Delta'$ define tropical cycles on $\Delta$ we assume $\Sigma=\Sigma''$ and $\Delta=\Delta'$ and continue to construct a Minkowski weight $f_*c$ on $\Delta$. 
For a $k$-dimensional cone $\delta\in\Delta$ we define
\begin{equation*}
f_*c(\delta)=\sum_{\sigma\mapsto\delta} [N^\delta:f(N^\sigma)]c(\sigma),
\end{equation*}
where the sum is over all $k$-dimensional cones of $\Sigma$ with image equal to $\delta$. A slight modification of the argument for the embedded case \cite[Prop. 2.25]{GKM09} shows that $f_*c$ is balanced. We define the cycle $f_*A\in Z_k(\Delta)$ as the tropical cycle defined by $f_*c$. This definition is independent of the choices of subdivisions, as can be seen as follows. The support of $f_*c$ is equal to the union of the images of the $k$-dimensional cones contained in $|A|$ on  which $f$ is injective, hence independent of all choices. Let us denote it  by $|f_*A|$. Outside a $(k-1)$-dimensional subset of $|f_*A|$, namely the image of the union of all cones of $\Sigma$ of dimension less or equal to $k$ whose image has dimension strictly less than $k$, every point of $|f_*A|$ has finitely many preimages in $|A|$, and they all lie in the relative interiors of $k$-dimensional cones of $\Sigma$ on which $f$ is injective. Thus, locally, the number of summands occurring in the definition of $f_*c$ does not depend on the chosen subdivisions. Since the occurring lattice indices are local in the same sense, the independence of $f_*A$ of all choices follows. It is immediate from the construction that pushing forward is linear in $A$.

Now suppose that $f$ is arbitrary, and $A$ is contained in $Z_k(\Star_\Sigma(\sigma))$ for some $\sigma\in\Sigma$. The restriction $f|_{\overline{\Star_\Sigma(\sigma)}}$ defines a dominant morphism to $\overline{\Star_\Delta(\delta)}$ for some $\delta\in\Delta$. Using the construction for the dominant case we define 
\begin{equation*}
f_*A\coloneqq \left(f|_{\overline{\Star_\Sigma(\sigma)}}\right)_{\hspace{-.2em}*}A\in Z_k(\Star_\Delta(\delta))\subseteq Z_k(\overline\Delta)
\end{equation*}
and call it the \emph{push-forward} of $A$. Extending by linearity the push-forward defines a graded morphism $f_*\colon Z_*(\overline\Sigma)\rightarrow Z_*(\overline\Delta)$.
\xqed{\lozenge}
\end{constr}

\begin{prop}
\label{ITprop:push-forwards are functorial}
Let $f\colon\overline\Sigma\rightarrow\overline\Delta$ and $g\colon\overline\Delta\rightarrow\overline\Theta$ be morphisms of weakly embedded extended cone complexes. Then
\begin{equation*}
(g\circ f)_*=g_*\circ f_*.
\end{equation*}
\end{prop}

\begin{proof}
By construction it suffices to prove that if $f$ and $g$ are dominant, and $A\in Z_k(\Sigma)$, then $(g\circ f)_*A=g_*(f_*A)$. But this works analogously to the embedded case \cite[Remark 1.3.9]{HannesDiss}.
\end{proof}

\subsection{Intersecting with Divisors}

Next we construct an intersection product $\psi\cdot A$ for a divisor $\psi$ and a tropical cycle $A\in Z_k(\Sigma)$. Note that when $\Sigma=\Sigma(X)$ for a toroidal embedding $X$, the divisor $\psi$ corresponds to a Cartier divisor $D$ on $X$ supported away from $X_0$. We can only expect an intersection product $D\cdot Z$ with a subvariety $Z$ of $X$ to be defined without passing to rational equivalence if $Z$ intersects $X_0$ non-trivially. In the tropical world this is reflected by the requirement of $A$ living in the finite part of $\overline\Sigma$.

\begin{constr}[Tropical intersection products]
\label{ITconstr:intersection products with tropical Cartier divisors}
Let $\Sigma$ be a weakly embedded cone complex, $A\in Z_k(\Sigma)$ a tropical cycle, and $\psi$ a  Cartier divisor on $\overline\Sigma$. We construct the intersection product $\psi\cdot A\in Z_{k-1}(\overline\Sigma)$. First suppose that $A$ is defined by a Minkowski weight $c\in \Mink_k(\Sigma)$. Writing $i_\rho\colon\Star(\rho)\rightarrow\overline\Sigma$ for the inclusion map for all rays $\rho\in\Sigma_{(1)}$, we define
\begin{equation*}
\psi\cdot A =\sum_{\rho\in\Sigma_{(1)}}\psi(u_\rho)[i^*_\rho c] \in Z_{k-1}(\overline\Sigma),
\end{equation*}
where $u_\rho$ is the primitive generator of $\rho$.

In general, $A$ is represented by a Minkowski weight on a proper subdivision $\Sigma'$ of $\Sigma$. Let $f\colon\overline\Sigma'\rightarrow\overline\Sigma$ be the morphism of weakly embedded extended cone complexes induced by the subdivision. Then we define 
\begin{equation*}
\psi\cdot A = f_*(\psi\cdot_{\Sigma'} A),
\end{equation*}
where $\psi\cdot_{\Sigma'}A$ denotes the cycle in $Z_{k-1}(\overline\Sigma')$ constructed above.

To verify that this is independent of the choice of $\Sigma'$ it suffices to prove that in case $A$ is already represented by a Minkowski weight on $\Sigma$, this gives the same definition as before. Let $c'\in\Mink_k(\Sigma')$ and $c\in\Mink_k(\Sigma)$ be the Minkowski weights representing $A$. For a ray $\rho\in\Sigma$, which is automatically a ray in $\Sigma'$, it is not hard to see that the complex $\Star_{\Sigma'}(\rho)$ is a proper subdivision of $\Star_\Sigma(\rho)$. Writing $i_\rho'\colon S_{\Sigma'}(\rho)\rightarrow \overline\Sigma'$ for the inclusion we see that $f_*[(i'_\rho)^* c']=[i_\rho^*c]$. For a ray $\rho'\in\Sigma'\setminus\Sigma$ there exists a minimal cone $\tau\in\Sigma$ containing it, which has dimension at least $2$. Every $k$-dimensional cone $\sigma'\in\Sigma'$ containing $\rho'$ with $c'(\sigma')\neq 0$ is contained in a $k$-dimensional cone $\sigma\in\Sigma$. Then $\tau$ must be a face of $\sigma$ and hence the image of $\sigma'/\rho'\subseteq\Star_{\Sigma'}(\rho')$ in $\Star_\Sigma(\tau)$ has dimension at most $\dim(\sigma/\tau)$, which is strictly less than $k$. This shows that  $f_*[(i_{\rho'}')^*c']=0$ and finishes the prove of the independence of choices.
\xqed{\lozenge}
\end{constr}

\begin{rem}
\label{ITrem:Computation of weights in intersection product}
In the construction of $\psi\cdot A$ for general $A$, represented by a Minkowski weight $c$ on the proper subdivision $\Sigma'$ of $\Sigma$, we can compute the contributions of the weights on the cones of $\Sigma'$ explicitly. Assume that $\sigma'$ is a $k$-dimensional cone of $\Sigma'$ and we want to determine its contribution to the component of $\psi\cdot A$ in $\Star_\Sigma(\tau)$ for some $\tau\in\Sigma$. We can only have such a contribution if $\sigma'$ has a ray $\rho\in\sigma'_{(1)}$ intersecting the relative interior of $\tau$. In this case, let $\delta\in\Sigma$ be the smallest cone of $\Sigma$ containing $\sigma'$. As $\sigma'\cap\relint(\tau)\neq\emptyset$ we must have $\tau\preceq\delta$. The image of $\sigma'/\rho\subseteq\Star_{\Sigma'}(\rho)$ in $\delta/\tau\subseteq\Star_\Sigma(\tau)$ is the image of $\sigma'$ under the canonical projection $\delta\rightarrow \delta/\tau$. This is $(k-1)$-dimensional if and only if the sublattice $N^{\sigma'}\cap N^\tau$ of $N^\delta$ has rank $1$. In this case, we have $\sigma'\cap\tau=\rho$, and the contribution of $\sigma'$ to the component of $\psi\cdot A$ in $\Star_\Sigma(\tau)$ is the cone $(\sigma'+\tau)/\tau$ with weight 
\begin{equation*}
\psi(u_\rho)\ind{N^\tau+N^{\sigma'}}c(\sigma'),
\end{equation*}
where $u_\rho$ is the primitive generator of $\rho$, and the occurring index is the index of $N^\tau+N^{\sigma'}$ in its saturation in $N^\delta$.

We see that $\sigma'$ contributes exactly to those boundary components of $\psi\cdot A$ which we would expect from topology. This is because ``the part at infinity'' $\overline\sigma'\setminus\sigma'$ of $\sigma'$ in $\overline\Sigma$ intersects $\Star_\Sigma(\tau)$ if and only if $\relint(\tau)\cap\sigma'\neq\emptyset$, and the intersection is $(k-1)$-dimensional if and only if $N^{\sigma'}\cap N^\tau$ has rank $1$.
\end{rem}

As mentioned earlier, Minkowski weights on the cone complex $\Sigma(X)$ of a toroidal embedding $X$ correspond to cocycles on $X$. For an (algebraic) Cartier divisor $D$ on $X$ supported away from $X_0$ and a cocycle $c$ in $A^*(X)$ the natural ``intersection'' is the cup product $c_1(\mathcal O_X(D))\cup c $. It turns out that we can describe the associated Minkowski weight of the product by a tropical cup product in case the associated tropical divisor of $D$ is combinatorially  principal (cf.\ Proposition \ref{RATprop:Cup products and Tropicalization commute}).

\begin{constr}[Tropical cup products]
\label{ITconstr:Cup product}
Let $\Sigma$ be a weakly embedded cone complex, $c\in \Mink_k(\Sigma)$ a Minkowski weight, and $\psi$ a \tlt{}-divisor on $\Sigma$. We construct a Minkowski weight $\varprod{\psi}{c}\in \Mink_{k-1}(\Sigma)$. For every $\sigma\in\Sigma$ choose an integral linear function $\psi_\sigma\in M^\Sigma$ such that  $\psi_\sigma\circ\phi_\Sigma|_\sigma =\psi|_\sigma$. Note that $\psi_\sigma$ is not uniquely defined by this property, but its restriction $\psi_\sigma|_{N^\Sigma_\sigma}$ is.
Whenever we have an inclusion $\tau\preceq\sigma$ of cones of $\Sigma$, the function $\psi_\sigma-\psi_\tau$ vanishes on $N^\Sigma_\tau$ and hence defines a morphism $N^\Sigma/N^\Sigma_\tau\rightarrow\Z$. Therefore, we can define a weight on the $(k-1)$-dimensional cones of $\Sigma$ by 
\begin{equation*}
\varprod{\psi}{c}\colon\tau\mapsto\sum_{\sigma:\tau\prec\sigma}(\psi_\sigma-\psi_\tau)(c(\sigma)u_{\sigma/\tau}),
\end{equation*}
where the sum is taken over all $k$-dimensional cones of $\Sigma$ containing $\tau$. The balancing condition for $c$ ensures that this weight is independent of the choices involved. It has been proven in \cite[Prop. 3.7a]{AR10} for the embedded case that $\psi\cup c$ is a Minkowski weight again, and a slight modification of the proof in loc.\ cit.\ works in our setting as well.

The cup-product can  be extended to apply to tropical cycles as well. For $A\in Z_k(\Sigma)$ there is a subdivision $\Sigma'$ of $\Sigma$ such that $A$ is represented by a Minkowski weight $c\in \Mink_k(\Sigma')$. It is easy to see that the tropical cycle $[\varprod{\psi}{c}]$ does not depend on the choice of $\Sigma'$. Thus, we can define $\varprod{\psi}{A}\coloneqq[\varprod{\psi}{c}]$. For details we again refer to \cite{AR10}.

It is immediate from the definition that the cup-product is independent of the linear equivalence class of the divisor. Therefore, we obtain a pairing
\begin{equation*}
\ClTLT(\Sigma)\times Z_*(\Sigma)\rightarrow Z_*(\Sigma),
\end{equation*}
which is easily seen to be bilinear. It can be extended to include tropical cycles in  the boundary of $\overline\Sigma$. Let $\tau\in\Sigma$, and denote the inclusion $i:\Star(\tau)\rightarrow\overline\Sigma$ by $i$. For a tropical cycle $A\in\Star_\Sigma(\tau)$ and a divisor class $\overline\psi\in\ClTLT(\Sigma)$ we define $\varprod{\overline\psi}{A}\coloneqq\varprod{i^*\overline\psi}{A}$. In this way we obtain a bilinear map
\begin{equation*}
\ClTLT(\Sigma)\times Z_*(\overline\Sigma)\rightarrow Z_*(\overline\Sigma). \xqedhere{4.29cm}{\lozenge}
\end{equation*}
\end{constr}

In algebraic geometry, intersecting a cycle with multiple Cartier divisors does not depend on the order of the divisors. We would like to prove an analogous result for tropical intersection products on cone complexes. However, to define multiple intersections we need rational equivalence, which we will only introduce in Definition \ref{ITdefn:Rational Equivalence}. For cup-products on the other hand there is no problem in defining multiple intersections and they do, in fact, not depend on the order of the divisors:

\begin{prop}
\label{ITprop:cup products commute}
Let $\Sigma$ be a weakly embedded cone complex, $A\in Z_*(\overline\Sigma)$ a tropical cycle, and $\psi,\chi\in\TLT(\Sigma)$ two \tlt{}-divisors. Then we have
\begin{equation*}
\varprod{\psi}{(\varprod{\chi}{A})}= \varprod{\chi}{(\varprod{\psi}{A})}.
\end{equation*}
\end{prop}

\begin{proof}
This can be proven similarly as in the embedded case \cite[Prop. 3.7]{AR10}.
\end{proof}

Next we will consider the projection formula for weakly embedded cone complexes. We will prove two versions, one for ``$\cdot$''-products and one for ``$\varprod{}{}$''-products.

\begin{prop}
\label{ITprop:Projection formula}
Let $f\colon\overline\Sigma\rightarrow\overline\Delta$ be a morphism between weakly embedded cone complexes.
\begin{enumerate}[label=\emph{\alph*})]
\item If $f$ is dominant, $A\in Z_*(\Sigma)$, and $\psi$ is a divisor on $\Delta$, then 
\begin{equation*}
f_*(f^*\psi\cdot A)= \psi\cdot f_* A.
\end{equation*}
\item For $\overline\psi\in\ClTLT(\Delta)$ and $A\in Z_*(\overline\Sigma)$ we have
\begin{equation*}
f_*(\varprod{f^*\overline\psi}{A})= \varprod{\overline\psi}{f_* A}.
\end{equation*}
\end{enumerate}
\end{prop}

\begin{proof}
In both cases both sides of the equality we wish to prove are linear in $A$. Therefore, we may assume that $A$ is a $k$-dimensional tropical cycle. For part a) choose proper subdivisions $\Sigma'$ and $\Delta'$ of $\Sigma$ and $\Delta$, respectively, such that $A$ is defined by a Minkowski weight $c\in \Mink_k(\Sigma')$ and $f$ maps cones of $\Sigma'$ onto cones of $\Delta'$. Consider the commutative diagram
\begin{center}
\begin{tikzpicture}[auto]
\matrix[matrix of math nodes, row sep=4ex, column sep= 3em]{
|(Sig')| \overline\Sigma' &
|(Delt')| \overline \Delta' \\
|(Sig)| \overline\Sigma &
|(Delt)| \overline\Delta.\\
};
\begin{scope}[->,font=\footnotesize]
\draw (Sig') --node{$f'$} (Delt');
\draw (Sig) --node{$f$} (Delt);
\draw (Sig') --node{$p$} (Sig);
\draw (Delt') --node{$q$} (Delt);
\end{scope}
\end{tikzpicture}
\end{center}
By construction of the intersection product, we have $f^*\psi\cdot A= p_*[f^*\psi\cdot c]$, where we identify $f^*\psi$ with $p^*(f^*\psi)$. Therefore, we have
\begin{equation*}
f_*(f^*\psi\cdot A)=(f\circ p)_*[f^*\psi\cdot c]=q_*\big(f'_*[f^*\psi\cdot c]\big).
\end{equation*}
On the other hand, by construction of push-forward and intersection product, we have
\begin{equation*}
\psi\cdot f_*A=q_*[\psi\cdot f'_*c].
\end{equation*}
This reduces the proof to the case where images of cones in $\Sigma$ are cones in $\Delta$, and $A$ is represented by a Minkowski weight $c\in \Mink_k(\Sigma)$. Let $\rho\in\Delta$ be a ray and $\delta\in\Delta$ a $k$-dimensional cone containing it. Every cone $\sigma\in\Sigma$ mapping injectively onto $\delta$ contains exactly one ray $\rho'$ mapping onto $\rho$. The contribution of the cone $\sigma$ in $f_*(i_{\rho'}^*c)$, where $i_{\rho'}\colon\Star_\Sigma(\rho')\rightarrow \overline\Sigma$ is the inclusion, is  a weight $[N^{\delta/\rho}: f(N^{\sigma/\rho'})]c(\sigma)$ on the cone $\delta/\rho$. The lattice index is equal to 
\begin{equation*}
[N^{\delta/\rho}: f(N^{\sigma/\rho'})]=
[N^\delta/N^\rho:(f(N^\sigma)+N^\rho)/N^\rho]=
[N^\delta:f(N^\sigma)]/[N^\rho:f(N^{\rho'})].
\end{equation*}
The index $[N^\rho:f(N^{\rho'})]$ is also defined by the fact that $f(u_{\rho'})=[N^\rho:f(N^{\rho'})]u_\rho$, where $u_{\rho}$ and $u_{\rho'}$ denote the primitive generators of $\rho$ and $\rho'$, respectively. Combined, we see that the weight of the component of $f_*(f^*\psi\cdot c)$ in $\Star_\Delta(\rho)$ at $\delta/\rho$ is equal to 
\begin{equation*}
\sum_{\rho'\mapsto\rho}\sum_{\substack{\thickspace\rho'\prec\sigma,\\ \sigma\mapsto\delta}}\psi(f(u_{\rho'}))[N^\delta:f(N^\sigma)]/[N^\rho:f(N^{\rho'})]c(\sigma)=
\psi(u_\rho)\sum_{\sigma\mapsto\delta}[N^\delta:f(N^\sigma)]c(\delta),
\end{equation*}
which is precisely the multiplicity of $\psi\cdot f_*c$ at $\delta/\rho$. Because the $\Star_\Delta(\delta)$-component is $0$ for both sides of the projection formula if $\delta$ is not a ray, we have proven part a).

For part b) we may reduce to the case that $f$ is dominant and $A\in Z_k(\Sigma)$. It also  suffices to prove the equation for divisors instead of divisor classes. Analogous to part a) we then reduce to the case where $A$ is given by a Minkowski weight on $\Sigma$, and images of cones of $\Sigma$ are cones of $\Delta$. In this case, a slight variation of the proof of the projection formula for embedded complexes \cite[Prop. 4.8]{AR10} applies.
\end{proof}

\begin{prop}
\label{ITprop:cup and cap product commute}
Let $\Sigma$ be a weakly embedded cone complex.
\begin{enumerate}[label=\emph{\alph*})]
\item For every Minkowski weight $c\in \Mink_k(\Sigma)$, \tlt{}-divisor $\psi\in\TLT(\Sigma)$, and cone $\tau\in\Sigma$ we have
\begin{equation*}
i^*(\varprod{\psi}{c})=\varprod{i^*\psi}{i^*c},
\end{equation*}
where $i\colon\Star(\tau)\rightarrow\overline\Sigma$ denotes the inclusion map.
\item
If $A\in Z_*(\Sigma)$, $\chi\in\Div(\Sigma)$, and $\overline\psi\in\ClTLT(\Sigma)$, then
\begin{equation*}
\varprod{\overline\psi}{(\chi\cdot A)}= \chi\cdot(\varprod{\overline\psi}{A}).
\end{equation*}
\end{enumerate}
\end{prop}

\begin{proof}
We begin with part a). If $\dim(\tau)\geq k$, then both sides of the equation are equal to $0$. So assume $\dim(\tau)< k$. Let $\sigma\in\Sigma$ be a $(k-1)$-dimensional cone containing $\tau$, and let $\gamma_1,\dots,\gamma_l$ be the $k$-dimensional cones of $\Sigma$ containing $\sigma$. To compute the weights of the two sides of the equation at $\sigma/\tau$, we may assume that $\psi$ vanishes on $\sigma$. Let $\psi_i\in M^\Sigma$ be linear functions such that $\psi_i\circ\phi_\Sigma|_{\gamma_i}=\psi|_{\gamma_i}$. Then if $v_1,\dots, v_l$ are representatives in $N^\Sigma$ of the lattice normal vectors $u_{\gamma_1/\sigma},\dots,u_{\gamma_l/\sigma}$, the weight of $\varprod{\psi}{c}$ at $\sigma$, and hence the weight of $i^*(\varprod{\psi}{c})$ at $\sigma/\tau$, is given by
\begin{equation*}
\sum_{i=1}^l\psi_i(v_i)c(\gamma_i).
\end{equation*}
Since $\psi$ vanishes on $\sigma$, each $\psi_i$ vanishes on $N^\Sigma_\tau$. Thus for every $i$ we have an induced map $\overline\psi_i\in \Hom(N^\Sigma/N^\Sigma_\tau,\Z)=M^{\Star(\tau)}$. By construction of the pullbacks of divisors, these functions define $i^*\psi$ around $\sigma/\tau$. We saw in Construction \ref{ITconstr:pulling back Minkowski weights} that the lattice normal vectors $u_{(\gamma_i/\tau)/(\sigma/\tau)}$ are the images of the lattice normal vectors $u_{\gamma_i/\sigma}$ under the quotient map $N^\Sigma/N^\Sigma_\tau\rightarrow N_{\phantom{\sigma/\tau}}^{\Star(\tau)}/N^{\Star(\tau)}_{\sigma/\tau}$. This implies that the image $\overline v_i$ of $v_i$ under the quotient map $N^\Sigma\rightarrow N^{\Star(\tau)}$ represents $u_{(\gamma_i/\tau)/(\sigma/\tau)}$. Hence, the weight of $\varprod{i^*\psi}{i^*c}$ at $\sigma/\tau$ is
\begin{equation*}
\sum_{i=1}^l \overline\psi_i(\overline v_i)i^*c(\gamma_i/\tau)=\sum_{i=1}^l \psi_i(v_i)c(\gamma_i),
\end{equation*}
proving the desired equality.

For part b) we my assume that $A$ is pure-dimensional. We first treat the case in which it is given by a Minkowski weight $c$ on $\Sigma$. Using the definition of the intersection product and the bilinearity of the cup-product we get
\begin{multline*}
\varprod{\overline\psi}{(\chi\cdot A)}=\left[\sum_{\rho\in\Sigma_{(1)}}\chi(u_\rho)\left(\varprod{\overline\psi}{i^*_\rho c}\right)\right] \overset{a)}{=} \\ 
=\left[\sum_{\rho\in\Sigma_{(1)}}\chi(u_\rho)i^*_\rho\left(\varprod{\overline\psi}{c}\right)\right]=\chi\cdot(\varprod{\overline\psi}{c}),
\end{multline*} 
where $i_\rho\colon\Star(\rho)\rightarrow\overline\Sigma$  denotes the inclusion map.
The general case follows by considering a suitable proper subdivision of $\Sigma$ and using the projection formulas.
\end{proof}

\subsection{Rational Equivalence}
\label{ITsubsec:Rational equivalence}

We denote by $\tR^1$ the weakly embedded cone complex associated to  $\Pj^1$ equipped with the usual toric structure. It is the unique complete fan in $\R$, its maximal cones being $\R\gz$ and $\R\lz$. We identify the extended cone complex $\ctR^1$ with $\R\cup\{\infty,-\infty\}$, where the boundary points $\infty$ and $-\infty$ of $\ctR^1$ correspond to the boundary points $0$ and $\infty$ of $\Pj^1$, respectively. The identity function on $\R$ defines a \tlt{}-divisor on $\tR^1$ which we denote by $\idR$. Its associated cycle $\idR\cdot[\ctR^1]$, where $[\ctR^1]$ is the $1$-cycle on $\tR^1$ with weight $1$ on both of its cones, is $[\infty]-[-\infty]$. We use $\ctR^1$ to define rational equivalence on cone complexes in the same way as $\Pj^1$ is used to define rational equivalence on algebraic varieties.

\begin{defn}
\label{ITdefn:Rational Equivalence}
Let $\Sigma$ be a weakly embedded cone complex, and let $p\colon\Sigma\times\tR^1\rightarrow \Sigma$ and $q\colon\Sigma\times\tR^1\rightarrow \tR^1$ be the projections onto the first and second coordinate. For $k\in\N$ we define $R_k(\overline\Sigma)$ as the subgroup of $Z_k(\overline\Sigma)$ generated by the cycles $p_*(q^*\idR\cdot A)$, where $A\in Z_{k+1}(\Star(\sigma\times 0))$ and $\sigma\in\Sigma$. We call two cycles in $Z_k(\overline\Sigma)$ \emph{rationally equivalent} if their difference lies in $R_k(\overline\Sigma)$.  The $k$-th \emph{(tropical) Chow group}  $A_k(\overline\Sigma)$ of $\overline\Sigma$ is defined as the group of $k$-cycles modulo rational equivalence, that is $A_k(\overline\Sigma)= Z_k(\overline\Sigma)/R_k(\overline\Sigma)$. We refer to the graded group $A_*(\overline\Sigma)=\bigoplus_{k\in\N} A_k(\overline\Sigma)$ as the (total) Chow group of $\overline\Sigma$. 
\end{defn}

\begin{rem}
\label{ITrem: Rational Equivalence is well-def}
Note that $q^*(\idR)\cdot A$ is strictly speaking not defined if $A$ is not contained in $Z_*(\Sigma\times\tR^1)=Z_*(\Star(0\times 0))$. However, as $q^*(\idR)$ vanishes on $\sigma\times 0$ for $\sigma\in\Sigma$, there is a canonically defined pullback to $\Star(\sigma\times 0)$ which we can intersect with cycles in $Z_*(\Star(\sigma\times 0))$. Taking the canonical identification of $\Star(\sigma\times 0)$ with $\Star(\sigma)\times \tR^1$, this linear function is nothing but the pull-back of $\idR$ via the projection onto $\tR^1$. So if we define $R_*(\Sigma)$ as the subgroup of $Z_*(\overline\Sigma)$ generated by $p_*(q^*\idR\cdot A)$ for $A\in Z_*(\Sigma\times\tR^1)$, then $R_*(\overline\Sigma)$ is generated by the push-forwards of $R_*(\Star(\sigma))$ for $\sigma\in\Sigma$. In particular, it follows immediately from the definition that the inclusions $i\colon\overline{\Star(\sigma)}\rightarrow\overline\Sigma$ induce push-forwards $i_*\colon A_*(\overline{\Star(\sigma)})\rightarrow A_*(\overline\Sigma)$ on the level of Chow groups.
\end{rem}

\begin{rem}
\label{ITrem:Tropical rational equivalence is analogous to algebraic rational equivalence}
Our definition of tropical rational equivalence is analogous to the definition of rational equivalence in algebraic geometry given for example in \cite[Section 1.6]{F98}: two cycles on an algebraic scheme $X$ are rationally equivalent if and only if they differ by a sum of elements of the form $p_*([q^{-1}\{0\}\cap V]- [q^{-1}\{\infty\} \cap V])$, where $p$ is the projection from $X\times \Pj^1$ onto $X$, and $V$ is a subvariety of $X\times\mathds P^1$ mapped dominantly to $\mathds P^1$ by the projection $q\colon X\times\Pj^1\rightarrow \Pj^1$. Rewriting $[q^{-1}\{0\}\cap V]- [q^{-1}\{\infty\} \cap V]$ as $q^*\divv(x)\cdot [V]$, where $x$ denotes the identity on $\Pj^1\setminus \{0,\infty\}=\mathds G_m$, we see that we obtain the tropical definition from the algebraic one by replacing $\Pj^1$ by its associated cone complex and $V$ by a cycle $A\in Z_*(\Star(\sigma\times 0))$. That we do not allow $A$ to be in $Z_*(\Star(\sigma\times \tau))$ for $\tau\neq 0$ corresponds to the dominance of $V$ over $\Pj^1$.
\end{rem}

\begin{prop}
\label{ITprop:Push-forwards pass to rational equivalence}
Let $f\colon\overline\Sigma\rightarrow \overline\Delta$ be a morphism between weakly embedded extended cone complexes. Then taking push-forwards passes to rational equivalence, that is $f_*$ induces a graded morphism $A_*(\overline\Sigma)\rightarrow A_*(\overline\Delta)$, which we again denote by $f_*$. 
\end{prop}

\begin{proof}
First assume that $f$ is dominant. Let $p$, $p'$ be the projections to the first coordinates of $\Sigma\times\tR$ and $\Delta\times\tR$, respectively, and similarly let $q$, $q'$ be the projections to the second coordinates. For $A\in Z_*(\Star_{\Sigma\times\tR^1}(\sigma\times 0))$ we have 
\begin{equation*}
f_*p_*(q^*\idR \cdot A)=
p'_* (f\times\id)_*\Big(\big((f\times\id)^*q'^*\idR\big)\cdot A\Big) =
p'_*\Big(q'^*\idR\cdot \big((f\times\id)_*A\big)\Big),
\end{equation*}
which is in $R_*(\overline\Delta)$. We see $R_*(\overline\Sigma)$ is mapped to $0$ by the composite $Z_*(\overline\Sigma)\rightarrow Z_*(\overline\Delta)\rightarrow A_*(\overline\Delta)$, which proves the assertion. 

If $f$ is not dominant, then $f$ can be factored as a dominant morphism followed by an inclusion $\overline{\Star_\Delta(\delta)}\rightarrow \overline\Delta$ for some $\delta$ in $\Delta$. For both of these morphisms the push-forward respects rational equivalence, and hence it does the same for $f$.
\end{proof}

Let $\tR^0$ denote the embedded cone complex associated to the trivial fan in the zero lattice $0$. It is the tropical analogue of the one point scheme $\Pj^0$ in the sense that it is the associated cone complex of $\Pj^0$, as well as in the sense of being the terminal object in the category of weakly embedded (extended) cone complexes. Note that $ \ctR^0=\tR^0$. Since all $1$-cycles on $\tR^0\times\tR^1=\tR^1$ are of the form $k\cdot [\ctR^1]$ for some $k\in\Z$ there is a natural identification $A_*(\ctR^0)=Z_*(\ctR^0)=\Z$.

\begin{defn}
Let $\Sigma$ be a weakly embedded cone complex and $f\colon\overline\Sigma\rightarrow\ctR^0$ the unique morphism to $\ctR^0$. We define the degree map $\deg\colon Z_*(\overline\Sigma)\rightarrow\Z$ as the composite of $f_*$ and the identification $Z_*(\ctR^0)=\Z$. By Proposition \ref{ITprop:Push-forwards pass to rational equivalence} the degree map respects rational equivalence, i.e.\ there is an induced morphism $A_*(\overline\Sigma)\rightarrow\Z$, which we again denote by $\deg$.
\end{defn}

\begin{prop}
\label{ITprop:cup and intersection product equal modulo ration equivalence}
Let $\Sigma$ be a weakly embedded cone complex, let $A\in Z_*(\Sigma)$, and let $\psi\in\TLT(\Sigma)$. Then 
\begin{equation*}
\varprod{\psi}{A}=\psi\cdot A \quad\text{in } A_*(\overline\Sigma).
\end{equation*}
\end{prop}

\begin{proof}
We may assume that $A$ is purely $k$-dimensional, as both sides are linear in $A$. We relate the intersection product and the cup product via the ``full graph'' of $\psi$, which has been used in the embedded case e.g.\ in \cite{Mik05,AR10}. Let $\Sigma'$ be a proper subdivision of $\Sigma$ such that $A$ is represented by a Minkowski weight $c$ on $\Sigma'$, and $\psi$ does not change sings on the cones of $\Sigma'$. Let $\Gamma_\psi\colon|\Sigma|\rightarrow |\Sigma\times\tR^1|$ be the graph map of $\psi$. Then the full graph $\Gamma_\psi(A)$ of $\psi$ on $A$ is the tropical cycle in $\Sigma\times\tR^1$ whose underlying subdivision consists of the subcones $\Gamma_\psi(\sigma')$ for $\sigma'\in\Sigma'_{(k)}$, and
\begin{align*}
\Gamma^+_{\leq\psi}(\tau)&= \{(x,t)\mid x\in\tau, t\leq \Gamma_\psi(x)\}\cap(\tau\times\R\gz)\quad\text{ and} \\
\Gamma^-_{\leq\psi}(\tau)&= \{(x,t)\mid x\in\tau, t\leq \Gamma_\psi(x)\}\cap(\tau\times\R\lz)
\end{align*}
for $\tau\in\Sigma'_{(k-1)}$. The weights of these cones are $c(\sigma')$ on $\Gamma_\psi(\sigma')$ and $(\varprod{\psi}{c})(\tau)$ on $\Gamma_{\leq\psi}^\pm(\tau)$. A slight modification of the argument given in \cite[Construction 3.3]{AR10} shows that this is a well-defined cycle, i.e.\ satisfies the balancing condition.

Let $p$ and $q$ be the projections onto the first and second coordinate of $\Sigma\times\tR^1$. We claim that 
\begin{equation*}
\label{ITequ:rational equivalence of cup and intersection product}
p_*\big( q^*\idR\cdot\Gamma_\psi(A)\big) = \psi\cdot A\ -\varprod{\psi}{A}.
\end{equation*}
To verify this let us analyze how the cones of $\Gamma_\psi(A)$ contribute to the left hand side of the equation using the computation of Remark \ref{ITrem:Computation of weights in intersection product}. It is easily seen that the cones of the form $\Gamma_{\leq\psi}^+(\tau)$ do not contribute at all. The cones $\Gamma_{\leq\psi}^-(\tau)$ on the other hand contribute to the $\Star(0\times\R\lz)$-component of $q^*\idR\cdot\Gamma_\psi(A)$ with weight $-\varprod{\psi}{c}(\tau)$ on the cone $\Gamma^-_{\leq\psi}(\tau)/(0\times \R\lz)$. This cone is mapped injectively onto $\tau$ by the isomorphism $\Star(0\times\R\lz)\rightarrow\Sigma$ induced by $p$. Hence, the total contribution of the cones considered so far is $-\varprod{\psi}{A}$. 

Now let $\sigma'\in\Sigma'_{(k)}$. Depending on whether $\psi$ is non-negative or non-positive on $\sigma'$, the cone $\Gamma_\psi(\sigma')$ can have contributions only in components corresponding to cones in $|\Sigma|\times\R\gz$ or $|\Sigma|\times\R\lz$. Without loss of generality assume $\psi$ is non-negative on $\sigma'$. For every $\tau\in\Sigma$ we have $\Gamma_\psi(\sigma')\cap(\tau\times\R\gz)=\Gamma_\psi(\sigma'\cap\tau)$. In particular, this is one dimensional with relative interior contained in $\relint(\tau\times\R\gz)$ if and only if $\rho\coloneqq\sigma'\cap\tau$ is one dimensional, intersects the relative interior of $\tau$, and $\psi|_\rho$ is nonzero. In this case let $\delta\in\Sigma$ be a cone containing $\sigma'$, and denote by $\psi_\delta\in M^\delta$ the linear function defining $\psi|_\delta$. We have the equality
\begin{equation*}
\Gamma_{\psi_\delta}(N^{\sigma'})+ (N^\tau\times\Z)=(N^{\sigma'}+N^\tau)\times\Z
\end{equation*}
of sublattices of $N^{\delta\times\R\gz}=N^\delta\times\Z$. Therefore, $\Gamma_\psi(\sigma')$ has a contribution in the $\Star(\tau\times\R\gz)$-component of $q^*\idR\cdot \Gamma_\psi(A)$ if and only $\sigma'$ has a contribution in the $\Star(\tau)$-component of $\psi\cdot A$. To see that they even contribute with the same weight we first notice the equality
\begin{equation*}
\ind{N^{\Gamma_\psi(\sigma')}+N^{\tau\times\R\gz}}=\ind{\Gamma_{\psi_\delta}(N^{\sigma'})+(N^\tau\times\Z)}=\ind{N^\tau+N^{\sigma'}}
\end{equation*}
of indices. Since the primitive generator $u_{\Gamma_\psi(\rho)}$ of $\Gamma_\psi(\sigma')\cap(\tau\times\R\gz)$ is equal to the image $\Gamma_\psi(u_\rho)$ of the primitive generator $u_\rho$ of $\rho$, we also see that
\begin{equation*}
q^*\idR(u_{\Gamma_\psi(\rho)})=q(u_\rho,\psi(u_\rho))=\psi(u_\rho).
\end{equation*}
Finally, the weight of $\Gamma_\psi(A)$ on $\Gamma_\psi(\sigma')$ is equal to that of $A$ on $\sigma'$ by definition. Hence, the contributions are equal. Combining this with the fact that  $(\Gamma_\psi(\sigma')+\tau\times\R\gz)/(\tau\times\R\gz)$ is mapped injectively onto $(\sigma'+\tau)/\tau$ by the isomorphism $\Star(\tau\times\R\gz)\rightarrow \Star(\tau)$ induced by $p$,
 we conclude that the total contribution of the cones of the form $\Gamma_\psi(\sigma')$ for $\sigma'\in\Sigma'_{(k)}$ to $p_*\big( q^*\idR\cdot\Gamma_\psi(A)\big)$ is $\psi\cdot A$, and with this we have proven the desired equality.
\end{proof}

\begin{prop}
\label{ITprop:cup product passes to rational equivalence}
Let $\Sigma$ be a weakly embedded cone complex. Then there is a well-defined bilinear map
\begin{equation*}
\ClTLT(\Sigma)\times A_*(\overline\Sigma)\rightarrow A_*(\overline\Sigma),\quad (\overline\psi,[A])\mapsto [\varprod{\psi}{A}].
\end{equation*}
By abuse of notation we denote this paring by ``$\cdot$''.
\end{prop}

\begin{proof}
We need to show that $\psi\cup A=0$ in $A_*(\overline\Sigma)$ for all $\psi\in\TLT(\Sigma)$ and $A\in R_*(\overline\Sigma)$. It suffices to show this for $A=p_*(q^*\idR\cdot B)$ for some $B\in Z_*(\Sigma\times\tR^1)$, where $p$ and $q$ denote the projections to the first and second coordinate. In this case we have 
\begin{equation*}
\varprod{\psi}{A}= p_*\big(\varprod{p^*\psi}{(q^*\idR\cdot A)}\big)=
p_*\big(q^*\idR\cdot(\varprod{p^*\psi} A)\big)\in R_*(\overline\Sigma),
\end{equation*}
where the first equality uses the projection formula, and the second one uses Proposition \ref{ITprop:cup and cap product commute}.
\end{proof}

\begin{rem}
Recall that the intersection product $\psi\cdot A$ of Construction \ref{ITconstr:intersection products with tropical Cartier divisors} is only defined when $A$ is contained in the finite part of $\overline\Sigma$. That the bilinear pairing of the preceding proposition is its appropriate extension is justified by Proposition \ref{ITprop:cup and intersection product equal modulo ration equivalence}, which states that modulo rational equivalence the cup-product and the ``$\cdot$''-product are essentially the same.
\end{rem}

\begin{defn}
\label{ITdefn:iterated intersections}
Let $\Sigma$ be a weakly embedded cone complex, let $\psi_1,\dotsc,\psi_k\in\ClTLT(\Sigma)$, and let $A\in A_*(\overline\Sigma)$. Then we  define the cycle class
\begin{equation*}
\psi_1\dotsm \psi_k\cdot A =\prod_{i=1}^k\psi_i \cdot A 
\end{equation*}
inductively by $\prod_{i=1}^k\psi_i \cdot A = \psi_1\cdot\left(\prod_{i=2}^k\psi_i \cdot A\right)$, where the base case is the pairing of Proposition \ref{ITprop:cup product passes to rational equivalence}.
\end{defn}

\section{Tropicalization}
\label{RATsec}

In this section we will tropicalize cocycles and cycles on toroidal embeddings. As already mentioned in Remark \ref{ITrem:Analogy between tropical and algebraic IT}, cocycles will define Minkowski weights, whereas cycles will define tropical cycles on the associated weakly embedded extended cone complex. Afterwards, we will show that tropicalization respects intersection theoretical constructions like push-forwards, intersections with Cartier divisors supported on the boundary, and rational equivalence. For the rest of this paper we assume that all algebraic varieties are defined over an algebraically closed field $k$ of characteristic $0$.

\subsection{The associated Minkowski Weight of a Cocycle}

In the well-known paper \cite{FS97}, cocycles on complete toric varieties are described by Minkowski weights on the associated fans by recording the degrees of the intersections with the boundary strata. The same method yields Minkowski weights associated to cocycles on complete toroidal embeddings:

\begin{prop}
\label{RATprop:tropicalization of cocycle is balanced}
Let $X$ be a complete $n$-dimensional toroidal embedding, and let $c\in A^k(X)$ be a cocycle. Then the $(n-k)$-dimensional weight 
\begin{equation*}
\omega\colon\Sigma(X)_{(n-k)} \rightarrow \Z,\;\;  \sigma\mapsto \deg(c\cap [\cOrb(\sigma)])
\end{equation*}
satisfies the balancing condition.
\end{prop}

\begin{proof}
Let $\tau$ be an $(n-k-1)$-dimensional cone of $\Sigma(X)$, and let $\sigma_1,\dots,\sigma_n$ be the $(n-k)$-dimensional cones containing it. To prove that the balancing condition is fulfilled at $\tau$, it suffices to show that
\begin{equation*}
\sum_{i=1}^n \omega(\sigma_i) \langle f, u_{\sigma_i/\tau}\rangle =\langle f, \sum_{i=1}^n \omega(\sigma_i)u_{\sigma_i/\tau} \rangle = 0
\end{equation*}
for every $f\in (N^X/N^X_\tau)^*=(N^X_\tau)^\perp$. A rational function $f\in M^X$ is contained in  $(N^X_\tau)^\perp$ if and only if it is an invertible regular function on $X(\tau)$.  Therefore, for every $f\in (N^X_\tau)^\perp$ we have $\divv(f)|_{X(\sigma_i)}\in M^{\sigma_i}\cap\tau^\perp$. The lattice $M^{\sigma_i}\cap\tau^\perp$ is the group of integral linear functions of the ray $\sigma_i/\tau$, and $u_{\sigma_i/\tau}$ is by definition the image of the primitive generator of $\sigma_i/\tau$ in $N^X/N^X_\tau$. It follows that $\langle f, u_{\sigma_i/\tau}\rangle$ is the pairing of $\divv(f)|_{X(\sigma_i)}$ with the primitive generator of $\sigma_i/\tau$. By Lemma \ref{ITlem: complex of stratum} this is equal to the paring of $\divv(f)|_{\cOrb(\tau)\cap X(\sigma_i)}$ with the primitive generator of the ray $\sigma_{\cOrb(\tau)}^{\Orb(\sigma_i)}$, which in turn is equal to the multiplicity of $\divv(f)|_{\cOrb(\tau)}$ at $\cOrb(\sigma_i)$. Using this we obtain
\begin{multline*}
\sum_{i=1}^n \omega(\sigma_i) \langle f, u_{\sigma_i/\tau}\rangle
=\sum_{i=1}^n \deg\big(c\cap [\cOrb(\sigma_i)]\big)\langle f, u_{\sigma_i/\tau}\rangle=\\
=\deg\left(c\cap \sum_{i=1}^n\langle f,u_{\sigma_i/\tau}\rangle[\cOrb(\sigma_i)]\right)
=\deg\left(c\cap\big[\divv(f)|_{\cOrb(\sigma_i)}\big]\right)=0,
\end{multline*}
which finishes the proof.
\end{proof}

\begin{defn}
Let $X$ be a complete $n$-dimensional toroidal embedding with weakly embedded cone complex $\Sigma$. We define the tropicalization map
\begin{equation*}
\Trop_X\colon A^k(X)\rightarrow \Mink_{n-k}(\Sigma),
\end{equation*}
by sending a cocycle on $X$ to the weight defined as in the preceding proposition. This is clearly a morphism of abelian groups. For a cone $\sigma\in\Sigma$, and a cocycle $c\in A^k(\cOrb(\sigma))$, we define $\Trop_X(c)\in \Mink_{n-\dim(\sigma)-k}(\Star_\Sigma(\sigma))$ as the pushforward to $\Mink_*(\Star_\Sigma(\sigma))$ of the Minkowski weight $\Trop_{\cOrb(\sigma)}(c)\in \Mink_*(\Sigma(\cOrb(\sigma)))$. We will just write $\Trop(c)$ when no confusion arises.
\end{defn}

\begin{prop}
\label{RATprop:tropicalization of pullback of cocycle to boundary}
Let $X$ be a complete toroidal embedding with weakly embedded cone complex $\Sigma$, and let $\tau\in\Sigma$. Furthermore, let $c\in A^*(X)$ be a cocycle on $X$, and let $i\colon\cOrb(\tau)\rightarrow X$ and $j\colon\overline{\Star(\tau)}\rightarrow\overline\Sigma$ be the inclusion maps. Then we have the equality
\begin{equation*}
\Trop_X(i^*c)=j^*\Trop_X(c)
\end{equation*}
of Minkowski weights on $\Star(\tau)$.
\end{prop}
\begin{proof}
By Lemma \ref{ITlem: complex of stratum} we have $\Orb(\sigma/\tau)=\Orb(\sigma)$, hence $\cOrb(\sigma/\tau)=\cOrb(\sigma)$. The projection formula implies that the weight of $\Trop(i^*c)$ at $\sigma/\tau$ is equal to the weight of $\Trop(c)$ at $\sigma$, which is equal to the weight of $j^*\Trop(c)$ at $\sigma/\tau$ by Construction \ref{ITconstr:pulling back Minkowski weights}.
\end{proof}

\begin{rem}
If the canonical morphism $\Sigma(\cOrb(\tau))\rightarrow \Star(\tau)$ is an isomorphism of weakly embedded cone complexes, the statement of Proposition \ref{RATprop:tropicalization of pullback of cocycle to boundary} reads
\begin{equation*}
\Trop(i^*c)=\Trop(i)^*\Trop(c).
\end{equation*} 
However, in general there may be more invertible regular functions on $\Orb(\tau)$ than obtained as restrictions of rational functions in $(N_\tau^X)^\perp\subseteq M^X$. In this case, the pullback of a Minkowski weight on $\Sigma$ to $\Star(\tau)$ may not be a Minkowski weight on $\Sigma(\cOrb(\tau))$. In other words, the pullback 
$\Trop(i)^*\colon\Mink_*(\Sigma(X))\rightarrow \Mink_*(\Sigma(\cOrb(\tau)))$ may be ill-defined. An example where this happens is when $X$ is equal to the blowup of $\Pj^2$ in the singular point of a plane nodal cubic $C$, and its boundary is the union of the exceptional divisor $E$ and the strict transform $\widetilde C$. The cone complex of $X$ consists of two rays $\rho_{\widetilde C}$ and $\rho_E$, corresponding to the two boundary divisors, which span two distinct strictly simplicial cones, one for each point in $E\cap \widetilde C$. The lattice $N^X$ is trivial, so that any two integers on the $2$-dimensional cones define a Minkowski weight in $\Mink_2(\Sigma(X))$. However, the pullback of such a weight to $\Star_{\Sigma(X)}(\rho_E)$ defines a Minkowski weight in $\Mink_1(\cOrb(\rho_E))$ only if the weights on its cones are equal. This is because the toroidal embedding $\Orb(\rho_E)\subseteq \cOrb(\rho_E)= E$ is isomorphic to $\Pj^1$ with two points in the boundary, the weakly embedded cone complex of which consists of two rays embedded into $\R$.
\end{rem}

\subsection{Tropicalizations of Cycles on Toroidal Embeddings}

Let $X$ be a toroidal embedding with weakly embedded cone complex $\Sigma$. In \cite{Uli13} Martin Ulirsch constructs a map $\trop^{\an}_X\colon X^\beth\rightarrow\overline\Sigma$ as a special case of his tropicalization procedure for fine and saturated logarithmic schemes. We recall that, since $X$ is separated, $X^\beth$ is the analytic domain of the Berkovich analytification $X^{\an}$ \cite{Berk90} consisting of all points which can be represented by an $R$-integral point for some rank one valuation ring $R$ extending $k$. The notation is due to Thuillier \cite{Thu07}. The map $\trop_X^{\an}$ restricts to a map $X_0^{\an}\cap X^\beth\rightarrow\Sigma$ generalizing the "$\ord$"-map $X_0\big(k\big(\!(t)\!\big)\!\big)\cap X\big(k[\mspace{-2mu}[t]\mspace{-2mu}]\big)\rightarrow\Sigma$ from \cite{KKMSD73}. Ulirsch's tropicalization map is also compatible with Thuillier's retraction map $\mathbf p_X\colon X^\beth\rightarrow X^\beth$ \cite{Thu07} in the sense that the retraction factors through $\trop^{\an}_X$. The tropicalization map allows to define the set-theoretic tropicalization of a subvariety $Z$ of $X_0$ as the subset $\trop_X(Z)\coloneqq\trop^{\an}_X(Z^{\an}\cap X^\beth)$ of $\Sigma$. This set has been studied in greater detail in \cite{Uli15}, where several parallels to tropicalizations of subvarieties of tori are exposed: first of all, $\trop_X(Z)$ can be given the structure of an at most $\dim(Z)$-dimensional cone complex whose position in $\Sigma$ reflects the position of $Z$ in $X$. Namely, there is a toroidal version of Tevelev's Lemma \cite[Lemma 2.2]{Tev07} stating that $\trop_X(Z)$ intersects the relative interior of a cone $\sigma\in\Sigma$ if and only if $\overline Z$ intersects $\Orb(\sigma)$. Furthermore, if $\trop_X(Z)$ is a union of cones of $\Sigma$, then $\overline Z$ intersects all strata properly. When $X$ is complete, we use this to give $\trop_X(Z)$ the structure of a tropical cycle in a similar way as done in \cite{ST08} in the toric case. First, we choose a simplicial proper subdivision $\Sigma'$ of $\Sigma$ such that $\trop_X(Z)$ is a union of its cones. Since $\charak k=0$ by assumption and $\Sigma'$ is simplicial, the toroidal modification $X'=X'\times_\Sigma\Sigma'$ is the coarse moduli space of a smooth Deligne-Mumford stack \cite[Thm. 3.3]{Iwan09} and thus has an intersection product on its Chow group $A_*(X)_\Q$ with rational coefficients \cite{Vistoli}. 

\begin{defn}
Let $d=\dim(Z)$ and $\sigma\in \Sigma'_{(d)}$. We define the multiplicity  $\mult_Z(\sigma)$ of the cone $\sigma$ by $\deg([\overline Z']\cdot [\cOrb(\sigma)])$, where $\overline Z'$ denotes the closure of $Z$ in $X'$.
\end{defn}

\begin{rem}
The multiplicity $\mult_Z(\sigma)$ is independent of the rest of the cone complex $\Sigma'$. This is because $\overline Z'$ intersects all strata properly so that $[\overline Z']\cdot[\cOrb(\sigma)]$ is a well-defined $0$-cycle supported on $\overline Z'\cap \Orb(\sigma)$ and only depending on $X'(\sigma)$.
\end{rem}

\begin{prop}
\label{RATprop:multiplicities are well-def}
let $\Delta$ be a simplicial proper subdivision of $\Sigma'$, and let $\delta\in \Delta_{(d)}$ be a $d$-dimensional cone contained in $\sigma\in\Sigma'_{(d)}$. Then $\mult_Z(\delta)=\mult_Z(\sigma)$. In particular, we have $\mult_Z(\sigma)\in\Z$ for all $\sigma\in\Sigma'$. Furthermore, the weight 
\begin{equation*}
\Sigma'_{(d)}\rightarrow\Z,\;\;\sigma\mapsto \mult_Z(\sigma)
\end{equation*}
is balanced and the tropical cycle associated to it is independent of the choice of $\Sigma'$.
\end{prop}

\begin{proof}
To ease the notation, we replace $X$ and $X'$ by $X\times_\Sigma\Sigma'$ and $X\times_\Sigma\Delta$, respectively, and denote the closures of $Z$ in $X$ and $X'$ by $\overline Z$ and $\overline Z'$. Furthermore, we denote the toroidal modification induced by the subdivision $\Delta\rightarrow\Sigma'$ by  $f\colon X'\rightarrow X$. Since $X$ is an Alexander scheme in the sense of \cite{Vistoli}, there is a pullback morphism $f^*\colon A_*(X)_\Q\rightarrow A_*(X')_\Q$. The pullback $f^*[\overline Z]$ is represented by a cycle $[X']\cdot_f [\overline Z]\in A_d(f^{-1}\overline Z)_\Q$. Since $\overline Z$ intersects all strata properly, and $f$ locally looks like a toric morphism, the preimage $f^{-1}\overline Z$ is $d$-dimensional and $\overline Z'$ is its only $d$-dimensional component. Hence, $f^*[\overline Z]$ is a multiple of $[\overline Z']$. By the projection formula we have
\begin{equation*}
f_*f^*[\overline Z]=f_*(f^*[\overline Z]\cdot [X'])=[\overline Z]\cdot [X]=[\overline Z],
\end{equation*} 
showing that we, in fact, have $f^*[\overline Z]=[\overline Z']$. Again using the projection formula we obtain
\begin{multline*}
\mult_Z(\delta)=\deg([\overline Z']\cdot [\cOrb(\delta)])=\deg(f_*(f^*[\overline Z]\cdot[\cOrb(\delta)]))=\\
=\deg([\overline Z]\cdot f_*[\cOrb(\delta)])=\deg([\overline Z]\cdot [\cOrb(\sigma)])=\mult_Z(\sigma).
\end{multline*}
Since every complex has a strictly simplicial proper subdivision, we can choose $\Delta$ to be strictly simplicial. In this case, the multiplicity $\mult_Z(\delta)$ is defined by the ordinary intersection product on $A_*(X')$, and hence we have $\mult_Z(\sigma)=\mult_Z(\delta)\in\Z$. Similarly, to prove the balancing condition for the weight $\Sigma_{(d)}\ni\sigma\mapsto \mult_Z(\sigma)\in\Z$ it suffices to show that the induced weight on a strictly simplicial proper subdivision $\Delta$ is balanced. By what we just saw, this weight is equal to $\Delta_{(d)}\ni \delta\mapsto\mult_Z(\delta)\in\Z$. But this is nothing but the tropicalization of the cocycle corresponding to $[\overline Z']$ by Poincaré duality, which is balanced by Proposition \ref{RATprop:tropicalization of cocycle is balanced}. That the tropical cycle it defines is independent of all choices follows immediately from the fact that any two subdivisions have a common refinement.
\end{proof}

The previous result allows us to assign a tropical cycle to the subvariety $Z$ of $X_0$. Its support will be contained in the set-theoretic tropicalization of $Z$.

\begin{defn}
Let $X$ be a complete toroidal embedding with weakly embedded cone complex $\Sigma$, and let $Z\subseteq X_0$ be a $d$-dimensional subvariety. We define the \emph{tropicalization} $\Trop_X(Z)\in Z_d(\Sigma)$ as the tropical cycle represented by the tropicalization of the cocycle corresponding to the closure $\overline Z'$ of $Z$ in a smooth toroidal modification $X'$ of $X$ in which $\overline Z'$ intersects all strata properly. This is well-defined by Proposition \ref{RATprop:multiplicities are well-def}. Extending by linearity, we obtain a tropicalization morphism
\begin{equation*}
\Trop_X\colon Z_*(X)=\bigoplus_{\sigma\in\Sigma}Z_*(\Orb(\sigma))\xrightarrow{\bigoplus_{\sigma\in\Sigma}\Trop_{\cOrb(\sigma)}}
\bigoplus_{\sigma\in\Sigma}Z_*(\Star(\sigma))=Z_*(\overline\Sigma).
\end{equation*}
\end{defn}

\begin{rem}
Strictly speaking the "$\sigma$-th" coordinate of $\Trop_X$ is not $\Trop_{\smash{\cOrb(\sigma)}}$, but the composite $Z_*(\Orb(\sigma))\xrightarrow{\Trop_{\cOrb(\sigma)}}Z_*(\Sigma(\cOrb(\sigma)))\rightarrow Z_*(\Star(\sigma))$, where the second map is the push-forward induced by the identification of cone complexes of Lemma \ref{ITlem: complex of stratum}.
\end{rem}

\subsection{The Sturmfels-Tevelev Multiplicity Formula}
\label{RATsubsec:The Sturmfels-Tevelev multiplicity formula}

We now give a proof of the Sturmfels-Tevelev multiplicity formula \cite[Theorem 1.1]{ST08} in the toroidal setting. It has its origin in tropical implicitization and states that push-forward commutes with tropicalization. A version for the embedded case over fields with non-trivial valuation has been proven in \cite{BPR11}, \cite{OP13}, and \cite{Gub12}.

\begin{thm}
\label{RATthm:Push-forward and Tropicalization commute}
Let $f\colon X\rightarrow Y$ be a toroidal morphism of complete toroidal embeddings, and let $\alpha\in Z_*(X)$. Then 
\begin{equation*}
\Trop(f)_*\Trop_X(\alpha)= \Trop_Y(f_*\alpha).
\end{equation*}
\end{thm}
\begin{proof}
Let $\Sigma$ and $\Delta$ denote the weakly embedded cone complexes of $X$ and $Y$, respectively.
Since both sides of the equation are linear in $\alpha$, we may assume that $\alpha=[\overline Z]$, where $Z$ is a closed subvariety of $\Orb(\sigma)$ for some $\sigma\in\Sigma$, say of dimension $\dim(Z)=d$. The toroidal morphism $\cOrb(\sigma)\rightarrow X\rightarrow Y$ factors through a dominant toroidal morphism $\cOrb(\sigma)\rightarrow\cOrb(\delta)$ for some $\delta\in\Delta$, and since tropicalization and push-forward commute for closed immersions of closures of strata by definition, this allows us to reduce to the case where $f$ is dominant and $\Orb(\sigma)=X_0$. 

Assume that the dimension of $f(Z)$ is strictly smaller than that of $Z$.  Since $\Trop(f)(\trop_X(Z))$ is equal to $\trop_Y(f(Z))$ (this follows from the surjectivity of $Z^{\an}\rightarrow f(Z)^{\an}$ \cite[Prop.\ 3.4.6]{Berk90} and the functoriality of $\trop^{\an}$ \cite[Prop.\ 6.2]{Uli13}), this implies that $\Trop(f)$ is not injective on any facet of $\Trop_X(Z)$, hence 
\begin{equation*}
\Trop(f)_*\Trop_X([\overline Z])=0=\Trop_Y(0)=\Trop_Y(f_*[\overline Z]).
\end{equation*}
Now assume that $\dim(f(Z))=\dim(Z)=d$. We subdivide $\Sigma$ and $\Delta$ as follows. First, we take a proper subdivision $\Sigma'$ of $\Sigma$ such that $\trop_X(Z)$ is a union of cones of $\Sigma'$. Then we take a strictly simplicial proper subdivision $\Delta'$ of $\Delta$ such that the images of cones in $\Sigma'$ are unions of cones of $\Delta'$. Pulling back the cones of $\Delta'$, we obtain a proper subdivision $\Sigma''$ of $\Sigma'$ whose cones map to cones in $\Delta'$. By successively subdividing along rays (cf.\ \cite[Remark 4.5]{AK00}), we can achieve a proper subdivision of $\Sigma''$ whose cones are simplicial and mapped to cones of $\Delta$ by $\Trop(f)$. After renaming, we see that there are proper subdivisions $\Sigma'$ and $\Delta'$ of $\Sigma$ and $\Delta$, respectively, such that 
 $\Sigma'$ is simplicial, $\Delta'$ is strictly simplicial, $\Trop(f)$ maps cones of $\Sigma'$ onto cones of $\Delta'$, and $\trop_X(Z)$ is a union of cones in $\Sigma'$.

Let $X'=X\times_\Sigma\Sigma'$ and $Y'=Y\times_\Delta\Delta'$ the corresponding toroidal modifications. Then the induced toroidal morphism $f'\colon X'\rightarrow Y'$ is flat by \cite[Remark 4.6]{AK00}.
As $\trop_Y(f(Z))$ is the image of $\trop_X(Z)$, it is a union of cones of $\Delta'$. By construction of the tropicalization, the weight of $\Trop_Y(f_*[\overline Z])$ at a $d$-dimensional cone $\delta\in\Delta'$ is equal to
\begin{equation*}
[K(\overline Z):K(f(\overline Z))]\deg([f'(\overline Z)]\cdot [\cOrb(\delta)])=\deg(f'_*[\overline Z']\cdot[\cOrb(\delta)])=\deg([\overline Z']\cdot f'^*[\cOrb(\delta)]),
\end{equation*}
where the second equality uses the projection formula. 
The irreducible components of $f'^{-1}\cOrb(\delta)$ are of the form $\cOrb(\sigma)$, where $\sigma$ is minimal among the cones of $\Sigma'$ mapping onto $\delta$. All of these cones $\sigma$ are $d$-dimensional, and the multiplicity with which $\cOrb(\sigma)$ occurs in $f'^{-1}\cOrb(\delta)$ is $[N^\delta:\Trop(f)(N^\sigma)]$ as we see by comparing with the toric case using local toric charts. Combining this, we obtain
\begin{equation*}
\deg([\overline Z']\cdot f'^*[\cOrb(\delta)])= \sum_{\sigma\mapsto\delta}[N^\delta:\Trop(f)(N^\sigma)]\deg([\overline Z']\cdot [\cOrb(\sigma)]),
\end{equation*}
where the sum runs over all $d$-dimensional cones of $\Sigma'$ mapping onto $\delta$. Since $\deg([\overline Z']\cdot [\cOrb(\sigma)])$ is the weight of $\Trop_X(Z)$ at $\sigma$, the right hand side of this equation is precisely the multiplicity of $\Trop(f)_*\Trop_X(Z)$ at $\delta$.
\end{proof}

\subsection{Tropicalization and Intersections with Boundary Divisors}
\label{RATsubsec:Tropicalization and intersections with boundary divisors}

Let $X$ be a toroidal embedding with weakly embedded cone complex $\Sigma$. By definition of $\Sigma$, the restriction of a Cartier divisor $D$ on $X$ which is supported away from $X_0$ to a combinatorial open subset $X(\sigma)$ for some $\sigma\in\Sigma$ is determined by an integral linear function on $\sigma$. Since the cones of $\Sigma$ and the combinatorial open subsets of $X$ are glued accordingly, $D$ defines a continuous function $\psi\colon |\Sigma|\rightarrow \R$ which is integral linear on all cones of $\Sigma$. Conversely, all such functions define a Cartier divisor on $X$ which is supported away from $X_0$. When $X$ is not toric, the Picard groups of its combinatorial opens need not be trivial. Hence, the restriction $D|_{X(\sigma)}$ is not necessarily of the form $\divv(f)|_{X(\sigma)}$ for some $f\in K(X)$. But if it is, the rational function $f$ must be regular and invertible on $X_0$, and by definition of the weak embedding we then have $\psi|_\sigma=f\circ\phi_X$. We see that $D$ is principal on the combinatorial opens of $X$ if and only if $\psi$ is combinatorially principal. By the same argument we see that another such divisor $D'$ with corresponding tropical \tlt{}-divisor $\psi'$ is linearly equivalent to $D$ if and only if $\psi$ and $\psi'$ are linearly equivalent tropical \tlt{}-divisors on $\Sigma$.	

\begin{defn}
Let $X$ be a toroidal embedding with weakly embedded cone complex $\Sigma$. We write $\Bound$  for the subgroup of $\Div(X)$ consisting of Cartier divisors which are supported on $X\setminus X_0$. We denote the tropical Cartier divisor on $\Sigma$ corresponding to $D\in\Bound$ by $\Trop_X(D)$. If $D$ is principal on all combinatorial opens, or equivalently, if $\Trop_X(D)\in\TLT(\Sigma)$, we say that $D$ is combinatorially principal (\tlt{}). We write $\TLT(X)$ for the group of \tlt{}-divisors on $X$. Furthermore, we write $\ClTLT(X)$ for the image of $\TLT(X)$ in $\Pic(X)$, and for $\mathcal L\in\ClTLT(X)$ we write $\Trop_X(\mathcal L)$ for its corresponding tropical divisor class in $\ClTLT(\Sigma)$. If no confusion arises, we usually omit the reference to $X$.
\end{defn}

\begin{lem}
\label{RATlem: pullback of divisor to stratum}
Let $f\colon X\rightarrow Y$ be a morphism of toroidal embeddings, and let  $\mathcal L\in\ClTLT(Y)$. Then $f^*\mathcal L\in \ClTLT(X)$, and 
\begin{equation*}
\Trop_X(f^*\mathcal L)=\Trop(f)^*(\Trop_Y(\mathcal L)).
\end{equation*} 
\end{lem}
\begin{proof}
It suffices to prove the lemma for dominant toroidal morphisms and closed immersions of closures of strata. The dominant case follows immediately from the definition of $\Trop(f)$. So we may assume that $X=\cOrb(\delta)$ for some $\delta\in\Delta=\Sigma(Y)$, and $f\colon \cOrb(\delta)\rightarrow Y$ is the inclusion.  Let $D\in\TLT(Y)$ be a representative for $\mathcal L$ and write $\psi=\Trop_Y(D)$. Since $D$ is in $\TLT(Y)$, there exists $g\in M^\Delta$ such that $\psi|_\delta=g\circ\phi_\Delta|_\delta$. The tropicalization of the divisor $D-\divv(g)$ is $\psi-g\circ\phi_\Delta$. This vanishes on $\delta$ by construction, which means that $D-\divv(g)$ is supported away from $\Orb(\delta)$. By Lemma \ref{ITlem: complex of stratum} the tropicalization of the restriction of $D-\divv(g)$ to $\cOrb(\delta)$ is given by the tropical divisor in $\Div(\Star(\delta))=\Div(\Sigma(\cOrb(\delta)))$ induced by $\psi-g\circ\phi_\Delta$. This finishes the proof because $(D-\divv(g))|_{\cOrb(\delta)}$ represents $f^*\mathcal L$, and  the divisor on $\Star(\delta)$ induced by $\psi-g\circ\phi_\Delta$ represents  $\Trop(f)^*\psi$ by Construction \ref{ITconstr:Pullbacks of Cartier divisors}.
\end{proof}

\begin{prop}
\label{RATprop:Cup products and Tropicalization commute}
Let $X$ be a complete toroidal embedding, let $\mathcal L\in\ClTLT(X)$, and let $c\in A^{k}(X)$. Then
\begin{equation*}
\Trop(c_1(\mathcal L)\cup c)=\varprod{\Trop(\mathcal L)}{\Trop(c)}.
\end{equation*}  
\end{prop}
\begin{proof}
Let $n=\dim(X)$. Furthermore, let $\tau$ be a $(n-k-1)$-dimensional cone of $\Sigma=\Sigma(X)$, and let $\psi\in\TLT(\Sigma)$ and $\psi_\tau\in\TLT(\Sigma(\cOrb(\tau)))$ be representatives for $\Trop(\mathcal L)$ and $\Trop(i)^*\Trop(\mathcal L)$, respectively, where $i\colon\cOrb(\tau)\rightarrow X$ is the inclusion map. By Lemma \ref{RATlem: pullback of divisor to stratum} the tropicalization of $i^*\mathcal L$ is represented by $\psi_\tau$, hence
\begin{equation*}
\sum_{\tau\prec\sigma} \psi_\tau(u_{\sigma/\tau})[\cOrb(\sigma)]= [i^*\mathcal L]=c_1(\mathcal L)\cap [\cOrb(\tau)] \quad\text{in } A_{n-k-1}(X).
\end{equation*}
Therefore, the weight of $\Trop(c_1(\mathcal L)\cup c)$ at $\tau$ is
\begin{equation*}
\deg((c_1(\mathcal L)\cup c) \cap [\cOrb(\tau)])=\sum_{\tau\prec\sigma}\psi_\tau(u_{\sigma/\tau})\deg(c\cap [\cOrb(\sigma)]).
\end{equation*}
Since $\deg(c\cap[\cOrb(\sigma)])$ is the weight of $\Trop(c)$ at $\sigma$, this is equal to the weight of $\varprod{\psi}{\Trop(c)}$ at $\tau$ by Construction \ref{ITconstr:Cup product}.
\end{proof}

\begin{thm}
\label{RATthm:Tropicalization and Intersections commute}
Let $X$ be a complete toroidal embedding and let $D\in\Bound$. Then for every subvariety $Z\subseteq X_0$ we have
\begin{equation*}
\Trop(D\cdot[\overline Z])=\Trop_X(D)\cdot \Trop(Z).
\end{equation*} 
\end{thm}

\begin{proof} 
First note that $D$ is supported on $X\setminus X_0$ and hence $D\cdot[\overline Z]$ is a well-defined cycle. Let $\Sigma'$ be a strictly simplicial proper subdivision of $\Sigma=\Sigma(X)$ such that $\trop(Z)$ is a union of its cones, and let $f\colon X'\coloneqq X\times_\Sigma\Sigma'\rightarrow X$ be the corresponding toroidal modification. The algebraic projection formula and Theorem \ref{RATthm:Push-forward and Tropicalization commute} imply that
\begin{equation*}
\Trop_X(D\cdot[\overline Z])=\Trop_X(f_*(f^*D\cdot[\overline Z']))=\Trop(f)_*\Trop_{X'}(f^*D\cdot[\overline Z']),
\end{equation*}
whereas the tropical projection formula and Lemma \ref{RATlem: pullback of divisor to stratum} imply that
\begin{multline*}
\Trop_X(D)\cdot\Trop_X(Z)=\Trop_X(D)\cdot(\Trop(f)_*\Trop_{X'}(Z))=\\
=\Trop(f)_*(\Trop_{X'}(f^*D)\cdot\Trop_{X'}(Z)).
\end{multline*}
This reduces to the case where $X=X'$ and we may assume that $\Sigma$ is strictly simplicial and $\overline Z$ intersects all boundary strata properly.

Denote $\Trop(D)$ by $\psi$, and let $\rho\in\Sigma_{(1)}$ be a ray of $\Sigma$. Since $X$ is smooth, the boundary divisor $\cOrb(\rho)$ is Cartier. Let $\cOrb(\rho)\cdot[\overline Z]= [\cOrb(\rho)\cap \overline Z]=\sum_{i=1}^k a_i [W_i]$.  Because $\overline Z$ intersects all strata of $X$ properly, each $W_i$ meets $\Orb(\rho)$. For the same reason every $W_i$ meets all strata of $\cOrb(\rho)$ properly. It follows that $\Trop_X(W_i)$ is represented by the Minkowski weight on $\Star(\rho)$ whose weight on a cone $\sigma/\rho$ is equal to $\deg([\cOrb(\sigma)]\cdot [W_i])$, where the intersection product is taken in $\cOrb(\rho)$. Denote by $i_\rho$ the inclusion of $\cOrb(\rho)$ into $X$. Then $\cOrb(\rho)\cdot\ [\overline Z]=i_\rho^*[\overline Z]$ in $A_*(\cOrb(\sigma))$, and hence the weight of $\Trop_X(\cOrb(\rho)\cdot[\overline Z])$ at $\sigma/\rho$ is equal to 
\begin{equation*}
\sum_{i=1}^k a_i\deg([\cOrb(\sigma)]\cdot [W_i])=\deg([\cOrb(\sigma)]\cdot i_\rho^*[\overline Z])=\deg([\cOrb(\sigma)]\cdot [\overline Z]),
\end{equation*}
where the first two intersection products are taken in $\cOrb(\sigma)$, whereas the last one is taken in $X$, and the last equality follows from the projection formula. Since $\deg([\cOrb(\sigma)]\cdot [\overline Z])$ is the weight of $\Trop_X(Z)$ at $\sigma$, this implies the equality
\begin{equation*}
\Trop_X(\cOrb(\rho)\cdot[\overline Z])=j_\rho^*\Trop_X(Z),
\end{equation*}
where $j_\rho\colon\Star(\rho)\rightarrow\overline \Sigma$ is the inclusion.
Together with the fact that the divisor $\cOrb(\rho)$ occurs in $D$ with multiplicity $\psi(u_\rho)$, where $u_\rho$ denotes the primitive generator of $\rho$, we obtain
\begin{equation*}
\Trop_X(D\cdot[\overline Z])=\sum_{\rho\in\Sigma_{(1)}}\psi(u_\rho)\Trop_X(\cOrb(\rho)\cdot[\overline Z])=\sum_{\rho\in\Sigma_{(1)}} \psi(u_\rho)j_\rho^*\Trop_X(Z),
\end{equation*}
which is equal to $\psi\cdot\Trop_X(Z)$ by Construction \ref{ITconstr:intersection products with tropical Cartier divisors}. 
\end{proof}

\subsection{Tropicalizing Cycle Classes}
As already pointed out in Remark \ref{ITrem:Tropical rational equivalence is analogous to algebraic rational equivalence}, rational equivalence for cycles on weakly embedded extended cone complexes is defined very similarly as rational equivalence in algebraic geometry. In fact, now that we know that tropicalization respects push-forwards and intersections with boundary divisors it is almost immediate that it respects rational equivalence as well. The only thing sill missing is to relate the weakly embedded cone complex $\Sigma(X\times Y)$ of the product of two toroidal embeddings  $X$ and $Y$ to the product $\Sigma(X)\times \Sigma(Y)$ of their weakly embedded cone complexes. First note that by combining local toric charts for $X$ and $Y$ it is easy to see that $X_0\times Y_0\subseteq X\times Y$ really is a toroidal embedding. With the same method we see that the cone complexes $\Sigma(X\times Y)$ and $\Sigma(X)\times \Sigma(Y)$ are naturally isomorphic, the isomorphism being the product $\Trop(p)\times\Trop(q)$ of the tropicalizations of the projections from $X\times Y$. That this even is an isomorphism of weakly embedded cone complexes, that is that $N^{X\times Y}\rightarrow N^X\times N^Y$ is an isomorphism as well, follows from a result by Rosenlicht which states that the canonical map 
\begin{equation*}
\Gamma(X_0,\mathcal O_X^*)\times\Gamma(Y_0,\mathcal O_Y^*)\rightarrow \Gamma(X_0\times Y_0,\mathcal O_{X\times Y}^*)
\end{equation*}
is surjective \cite[Section 1]{KKV89}.

\begin{prop}
\label{RATprop:Tropicalization respects rational equivalence}
Let $X$ be a complete toroidal embedding with weakly embedded cone complex $\Sigma$. Then the tropicalization $\Trop_X\colon Z_*(X)\rightarrow Z_*(\overline\Sigma)$ induces a morphism $A_*(X)\rightarrow A_*(\overline\Sigma)$ between the Chow groups, which we again denote by $\Trop_X$.
\end{prop}

\begin{proof}
By definition, the Chow group $A_k(X)$ is equal to the quotient of $Z_k(X)$ by the subgroup $R_k(X)$ generated by cycles of the form $p_*(q^*([0]-[\infty])\cdot [W])$, where $p$ and $q$ are the first and second projection from the product $X\times\Pj^1$, and $W$ is an irreducible subvariety of $X\times\Pj^1$ mapping dominantly to $\Pj^1$. Considering $\Pj^1$ with its standard toric structure, the projections are dominant toroidal morphisms. The boundary divisor $[0]-[\infty]$ is given by the identity on the cone complex $\tR^1$ of $\Pj^1$ when considering its natural identification with $\R$. This tropical divisor was denoted by $\idR$ in Subsection \ref{ITsubsec:Rational equivalence}. It follows that the tropicalization of the Cartier divisor $q^*([0]-[\infty])$ is equal to $\Trop(q)^*(\idR)$. Applying Theorems \ref{RATthm:Push-forward and Tropicalization commute} and \ref{RATthm:Tropicalization and Intersections commute} we obtain that
\begin{equation*}
\Trop_X(p_*(q^*([0]-[\infty])\cdot [W]))=\Trop(p)_*(\Trop(q)^*\idR\cdot \Trop_{X\times\Pj^1}(W)).
\end{equation*}
Noting that $\Trop(p)$ and $\Trop(q)$ are the projections from $\Sigma(X\times\Pj^1)=\Sigma\times\tR^1$, and that the dominance of $W\rightarrow\Pj^1$ implies $\Trop(W)\in Z_k(\Star_{\Sigma(X\times\Pj^1)}(\sigma\times 0))$ for some $\sigma\in\Sigma$, we see that we have an expression exactly as given for the generators of $R_k(\overline\Sigma)$ in Definition \ref{ITdefn:Rational Equivalence}. Thus, the tropicalization $\Trop_X$ maps $R_k(X)$ to zero in $A_k(\overline\Sigma)$ and the assertion follows.
\end{proof}

\begin{example}
Consider the toroidal embedding $X$ of Example \ref{ITexample: Toric variety and P2 without two hyperplanes} b), that is $X=\Pj^2$, and the boundary is the union $H_1\cup H_2$, where $H_i=V(x_i)$. Its weakly embedded cone complex $\Sigma$ is shown in Figure \ref{ITfig:Weakly embedded cone complex}. Let $L_1=V(x_1+x_2)$, and $L_2=V(x_1+x_2-x_0)$. The line $L_2$ intersects all strata properly as it does not pass through the intersection point $P=(1:0:0)$ of $H_1$ and $H_2$. Because $L_2$ intersects $H_1$ and $H_2$ transversally, this implies that $\Trop(L_2)$ is represented by the Minkowski weight on $\Sigma$ having weight $1$ on both of its rays. In contrast, $L_1$ passes through $P$. Its strict transform in the blowup of $X$ at $P$ intersects the exceptional divisor, and the intersection is in fact transversal, but none of the strict transforms of $H_1$ or $H_2$. The blowup at $P$ is the toroidal modification corresponding to the star subdivision of $\Sigma$. We see that $\Trop(L_2)$ is the tropical cycle given by a ray in the direction of $e_1+e_2$ with multiplicity $1$, where the $e_i$ are as in Example \ref{ITexample: Toric variety and P2 without two hyperplanes} b). By  Proposition \ref{RATprop:Tropicalization respects rational equivalence} the two tropical cycles $\Trop(L_1)$ and $\Trop(L_2)$ are rationally equivalent. The proposition also tells us that the tropicalization of the closure of the graph of $(x_1+x_2)/(x_1+x_2-x_0)$ in $X\times\Pj^1$ will give rise to the relation $\Trop(L_1)-\Trop(L_2)=0$ in $A_*(\overline\Sigma)$. Taking coordinates $x=x_1/x_0$ and $y=x_2/x_0$ on $X$, and a coordinate $z$ on $\Pj^1\setminus \{\infty\}$, the graph $W$ is given by the equation $W=V(xz+yz-z-x-y)$. Using these coordinates, it is easy to compute the classical tropicalization of $W$ when considered as subvariety of $\Pj^2\times\Pj^1$ with its standard toric structure. It is not hard to see that $\Trop_{X\times\Pj^1}(W)$ can be obtained from the classical tropicalization by intersecting with $(\R\gz)^2\times\R$. Its underlying set is depicted in Figure \ref{RATfig:rational equivalence}. The weights on its maximal faces are all $1$.
\end{example}

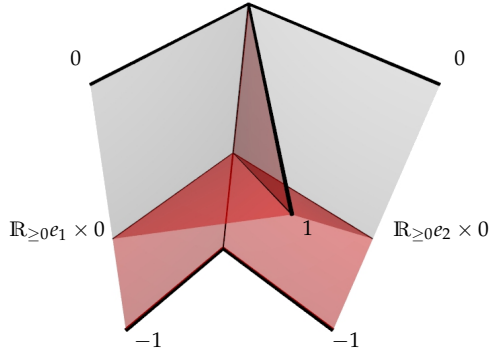
\begin{figure}
\centering
\begin{tikzpicture}[font=\footnotesize, scale=.5]
\node [anchor= south west] at (0,0) {\includegraphics[scale=.25]{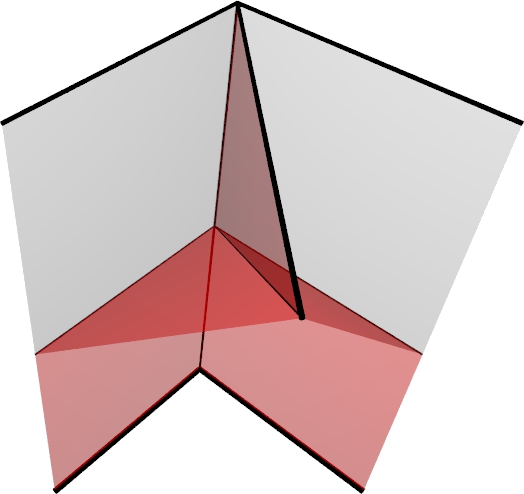}};
\node at (1.8,0)  {$-1$};
\node at (7,0)  {$-1$};
\node at (-.1,7.5) {$0$};
\node at (10,7.5) {$0$};
\node at (6,3) {$1$};
\node at (-.6,2.9) {$\R\gz e_1\times 0$};
\node at (8,2.9)[anchor=west] {$\R\gz e_2\times 0$};
\end{tikzpicture}
\caption{The tropicalization $\Trop_{X\times\Pj^1}(W)$ and its intersection with $\pr_2^*\idR$}
\label{RATfig:rational equivalence}
\end{figure}

Being able to tropicalize cycle classes we can formulate the compatibility statements of sections \ref{RATsubsec:The Sturmfels-Tevelev multiplicity formula} and \ref{RATsubsec:Tropicalization and intersections with boundary divisors} modulo rational equivalence. That the tropicalization of cycle classes commutes with push-forwards follows immediately from Theorem \ref{RATthm:Push-forward and Tropicalization commute}. For intersections with \tlt{}-divisors the compatibility is subject of the following result.

\begin{prop}
\label{RATprop:intersection of cycle class with divisor commutes with tropicalization}
Let $X$ be a complete toroidal embedding, let $\mathcal L\in \ClTLT(X)$, and let $\alpha\in A_*(X)$. Then we have 
\begin{equation*}
\Trop(c_1(\mathcal L)\cap \alpha)=\Trop(\mathcal L)\cdot\Trop(\alpha)
\end{equation*}
\end{prop}

\begin{proof}
Both sides are linear in $\alpha$, so we may assume that $\alpha=[Z]$ for some irreducible subvariety $Z$ of $X$. Let $\sigma\in\Sigma(X)$ be the cone such that $\Orb(\sigma)$ contains the generic point of $Z$, and denote by $i\colon \cOrb(\sigma)\rightarrow X$ the inclusion map. Furthermore, let $D\in\TLT(\cOrb(\sigma))$ be a representative of $i^*\mathcal L$. By Theorem \ref{RATthm:Push-forward and Tropicalization commute} and the projection formula we have
\begin{equation*}
\Trop_X(c_1(\mathcal L)\cap [Z])=\Trop(i)_*\Trop_{\cOrb(\sigma)}\left(c_1(i^*\mathcal L)\cap [Z]\right)=\Trop(i)_*\Trop_{\cOrb(\sigma)}(D\cdot [Z])
\end{equation*}
By Lemma \ref{RATlem: pullback of divisor to stratum} and Theorem \ref{RATthm:Tropicalization and Intersections commute} this is equal to 
\begin{equation*}
\Trop(i)_*\left(\big(\Trop(i)^*\Trop(\mathcal L)\big)\cdot\Trop_{\cOrb(\sigma)}\big(Z\cap \Orb(\sigma)\big)\right)
\end{equation*}Propositions
which is equal to $\Trop_X(\mathcal L)\cdot \Trop_X([Z])$ by the tropical projection formula (Proposition \ref{ITprop:Projection formula}) and the definition of the tropicalization.
\end{proof}

\subsection{Comparison with Classical Tropicalization}

Suppose we are given a complete toroidal embedding $X$, together with a closed immersion $\iota\colon X_0\rightarrow T$ into an algebraic torus $T$ with lattice of $1$-parameter subgroups $\Lambda$ and character lattice $\Lambvee$. Then we have two different ways to tropicalize $X_0$, either by taking the tropicalization $\Trop_X(X_0)=\Trop_X[X]$ in $\Sigma(X)$, or by taking the classical tropicalization in $\Lambda_\R$. Every morphism $\iota\colon X_0\rightarrow T$ induces a morphism 
\begin{equation*}
f\colon\Lambvee\rightarrow M^X,\;\; m\mapsto \iota^\sharp\chi^m,
\end{equation*}
whose dualization gives rise to a weakly embedded cone complex $\Sigma_\iota(X)$ with cone complex $\Sigma(X)$ and weak embedding $\phi_{X,\iota}=f_\R^*\circ\phi_X$. We define the tropicalization $\Trop_\iota(X_0)$ of $X_0$ with respect to $\iota$ as the classical tropicalization of the cycle $\iota_*[X_0]$ in $T$. Its underlying set $|\Trop_\iota(X_0)|$  is the image of the composite map
\begin{equation*}
X_0^{\an}\xrightarrow{\iota^{\an}} T^{\an}\xrightarrow{\trop_T^{\an}} \Lambda_\R,
\end{equation*}  
where $\trop_T^{\an}$ takes coordinate-wise minus-log-absolute values. As the next lemma shows, it is completely determined by $\Sigma_\iota(X)$:

\begin{lem}
\label{RATlem:tropicalization is given by image of cone complex}
Let $X$ be a toroidal embedding, and let $\iota\colon X_0\rightarrow T$ be a morphism into the algebraic torus $T$. Then the diagram
\begin{center}
\begin{tikzpicture}[auto]

\matrix[matrix of math nodes, row sep= 4ex, column sep= 3em, text height=1.5ex, text depth= .25ex]{
|(X0an)| X_0^{\an}\cap X^\beth  & |(SigX)| \Sigma_\iota(X) \\
|(Tan)| T^{\an} & |(LambR)| \Lambda_\R, \\
};

\begin{scope}[->,font=\footnotesize]
\draw (X0an)--node{$\trop_X^{\an}$}(SigX);
\draw (Tan)--node{$\trop_T^{\an}$} (LambR);

\draw (X0an) --node{$\iota^{\an}$} (Tan);
\draw (SigX) --node{$\phi_{X,\iota}$} (LambR);
\end{scope}

\end{tikzpicture}
\end{center}
is commutative. In particular, if $X$ is complete we have $|\Trop_\iota(X_0)|=\phi_{X,\iota}(|\Sigma_\iota(X)|)$.
\end{lem}

\begin{proof}
Let $x\in X_0^{\an}\cap X^\beth$, and let $\sigma\in\Sigma_\iota(X)$ be the cone such that the reduction $r(x)$, that is the image of the closed point of the spectrum of a rank-$1$ valuation ring $R$ under a representation $\Spec R\rightarrow X$ of $x$, is contained in $O(\sigma)$. By construction of $\trop_X^{\an}$ (cf.\ \cite[p. 424]{Thu07}, \cite[Def.\ 6.1]{Uli13}), the point $\trop_X^{\an}(x)$ is contained in $\sigma$ and its pairing with a divisor $D\in M^\sigma$ is equal to $-\log|f(x)|$ for an equation $f$ for $D$ around $r(x)$. It follows that the pairing of $\phi_{X,\iota}(\trop_X^{\an}(x))$ with $m\in\Lambvee$ is equal to $-\log|\iota^\sharp\chi^m(x)|$. This is clearly equal to $\langle m\,,\, \trop_T^{\an}(\iota^{\an}(x))\rangle$. The ``in particular'' statement follows, since for complete $X$ we have $X^\beth=X^{\an}$, and the tropicalization map $\trop_X^{\an}$ is always surjective \cite[Prop. 3.11]{Thu07}.
\end{proof}

\begin{rem}
The set-theoretic equality of the ``in particular''-part of the previous lemma is an instance of geometric tropicalization, that is a situation where one can read off the tropicalization of a subvariety of an algebraic torus from the structure of the boundary of a suitable compactification. It has already been pointed out in \cite{LQ11} how the methods for geometric tropicalization developed in \cite{HKT09} can be used to obtain the equality $|\Trop_\iota(X_0)|=\phi_{X,\iota}(|\Sigma_\iota(X)|)$ for a complete toroidal embedding $X$.
\end{rem}

The tropicalization $|\Trop_\iota(X_0)|$ is a union of cones and hence the underlying set of arbitrarily fine embedded cone complexes, that is fans, in $\Lambda_\R$. Similarly as in Construction \ref{ITconstr:push-forwards} there exists such a fan $\Delta$ and a proper subdivision $\Sigma'$ of $\Sigma_\iota(X)$ such that $\phi_{X,\iota}$ induces a morphism $\Sigma'\rightarrow \Delta$ of weakly embedded cone complexes. Therefore, there is a push-forward morphism
\begin{equation*}
Z_*(\Sigma(X))\rightarrow Z_*(\Sigma_\iota(X))=Z_*(\Sigma')\rightarrow Z_*(\Delta)=Z_*(\Trop_\iota(X_0)),
\end{equation*}
where $Z_*(\Trop_\iota(X_0))$ denotes the group of affine tropical  cycles in $\Trop_\iota(X_0)$  \cite[Def.\ 2.15]{AR10}. Again as in Construction \ref{ITconstr:push-forwards} we see that this morphism is independent of the choices of $\Sigma'$ and $\Delta$, and we denote it by $(\phi_{X,\iota})_*$.

\begin{thm}
\label{RATthm:tropicalization of image is pushforward}
Let $X$ be a complete toroidal embedding, and let $\iota\colon X\rightarrow T$ be a morphism into the algebraic torus $T$. Then for every subvariety $Z$ of $X_0$ we have
\begin{equation}
\label{RATequn:tropicalization of image is pushforward}
(\phi_{X,\iota})_*\Trop_X(Z)=\Trop_\iota(Z),
\end{equation}
where $\Trop_\iota(Z)$ denotes the classical tropicalization of the cycle $\iota_*[Z]$ in $T$.
\end{thm}

\begin{proof}
Since both sides of the equality are invariant under toroidal modifications, we may assume from the start that $\trop_X(Z)$ is a union of cones of $\Sigma(X)$. 
Let $d=\dim(Z)$. It follows from Lemma \ref{RATlem:tropicalization is given by image of cone complex} that both sides of (\ref{RATequn:tropicalization of image is pushforward}) vanish if the dimension of $\iota(Z)$ is strictly smaller than $d$. Hence we may assume that $\dim(\iota(Z))=d$.
Let $\Delta$ be a complete strictly simplicial fan in $\Lambda_\R$ such that $|\Trop_\iota(Z)|$, as well as $\phi_{X,\iota}(\sigma)$ for every $\sigma\in\Sigma_\iota(X)$, is a union of cones of $\Delta$. Similarly as described in Construction \ref{ITconstr:push-forwards} we obtain a simplicial proper subdivision $\Sigma'$ of $\Sigma_\iota(X)$ such that images of cones of $\Sigma'$ under $\phi_{X,\iota}$ are cones in $\Delta$. Namely, we can take suitable star-subdivisions of cones of the form $\phi_{X,\iota}^{-1}\delta\cap\sigma$ for $\delta\in\Delta$ and $\sigma\in\Sigma_\iota(X)$ along rays. Let $X'=X\times_{\Sigma(X)}\Sigma'$, and let $Y$ denote the toric variety associated to $\Delta$. Whenever $\delta\in\Delta$ and $\sigma'\in\Sigma'$ such that $\phi_{X,\iota}(\sigma')\subseteq \delta$, it follows form the definitions that $\divv(\iota^\sharp\chi^m)|_{X'(\sigma')}\in M_+^{\sigma'}$ for every $m\in \delta^\vee\cap \Lambvee$. In particular, $\iota^\sharp \chi^m$ is regular on $X'(\sigma')$, showing that $\iota$ extends to a morphism $X'(\sigma')\rightarrow U_\delta$, where $U_\delta$ is the affine toric variety associated to $U_\delta$. These morphisms glue and give rise to an extension $\iota'\colon X'\rightarrow Y$ of $\iota$. If $\sigma'\in\Sigma'$, and $\delta\in\Delta$ is minimal among the cones of $\Delta$ containing $\phi_{X,\iota}(\sigma')$, then  $\iota'(\Orb_{X'}(\sigma'))$ is contained in $\Orb_Y(\delta)$. Therefore, the preimages of torus orbits in $Y$ are unions of strata of $X$. Let $\delta\in\Delta_{(d)}$. By \cite[Lemma 2.3]{KP11}, the weight of $\Trop_\iota(Z)$ at $\delta$ is equal to $\deg([\cOrb_Y(\delta)]\cdot\iota'_*[\overline Z']) $, where we denote by $\overline Z'$ the closure of $Z$ in $X'$. Writing $D_1,\dots, D_d$ for the boundary divisors associated to the rays $\rho_1,\dotsc,\rho_d$ of $\delta$, this can be written as
\begin{equation*}
\deg(D_1\dotsm D_d\cdot \iota'_*[\overline Z'])=
\deg(\iota'^*D_1\dotsm\iota'^*D_d\cdot [\overline Z'])
\end{equation*}
by the projection formula. By the combinatorics of $\iota'$, the support of $\iota'^*D_i$ is the union of all strata $\Orb_{X'}(\sigma')$ associated to cones $\sigma'\in\Sigma'$ containing a ray which is mapped onto $\rho_i$. Hence, the intersection $|\iota'^*D_1|\cap\dotsc\cap|\iota'^*D_d|$ has pure codimension $d$ with components $\cOrb_{X'}(\sigma')$ corresponding to the $d$-dimensional cones of $\Sigma'$ mapping onto $\delta$. Let $\sigma'$ be such a cone. The  restriction $(\iota'^*D_i)|_{X'(\sigma')}$ is equal to $[N^\rho:\phi_{X,\iota}(N^{\rho'_i})]\cdot\cOrb_{X'}(\rho'_i)$, where $\rho'_i$ is the ray of $\sigma'$ mapping onto $\rho_i$. By comparing with a local toric model, we see that $\cOrb_{X'}(\rho'_1)\dotsm \cOrb_{X'}(\rho'_d)=\mult(\sigma')^{-1}[\cOrb_{X'}(\sigma')]$ \cite[Lemma 12.5.2]{CLS}, where the multiplicity $\mult(\sigma')$ is the index of the sublattice of $N^{\sigma'}$ generated by the primitive generators of the rays $\rho'_1,\dotsc,\rho'_d$ of $\sigma'$. Thus, the multiplicity of $[\cOrb_{X'}(\sigma')]$ in $\iota'^*D_1\dotsm\iota'^*D_d$ is equal to
\begin{equation*}
\frac 1{\mult(\sigma')}\prod_{i=1}^d[N^\rho:\phi_{X,\iota}(N^{\rho'_i})].
\end{equation*}
We easily convince ourselves that this is nothing but the index $[N^\delta:\phi_{X,\iota}(N^{\sigma'})]$. It follows that the weight of $\Trop_\iota(Z)$ at $\delta$ is equal to 
\begin{equation*}
\sum_{\sigma'\mapsto \delta}[N^\delta:\phi_{X,\iota}(N^{\sigma'})]\deg([\cOrb_{X'}(\sigma')]\cdot [\overline Z']),
\end{equation*}
where the sum is taken over all $d$-dimensional cones $\sigma'\in\Sigma'$ which are mapped onto $\delta$ by $\phi_{X,\iota}$. Since $\deg([\cOrb_{X'}(\sigma')]\cdot [\overline Z'])$ is the weight of $\Trop_X(Z)$ at $\sigma'$ by definition, the desired equality follows from the construction of the push-forward.
\end{proof}

\begin{cor}
\label{RATcor:weights on geometric tropicalization}
Let $X$ be a complete toroidal embedding, and let $\iota\colon X\rightarrow T$ be a closed immersion into the algebraic torus $T$. Then the weight of $\Trop_\iota(X_0)$ at a generic point $w$ of its support $\phi_{X,\iota}|\Sigma(X)|$ is equal to
\begin{equation*}
\sum_\sigma \ind{\phi_{X,\iota}(N^\sigma)},
\end{equation*}
where the sum is taken over all $\dim(X)$-dimensional cones of $\Sigma(X)$ whose image under $\phi_{X,\iota}$ contains $w$.
\end{cor}

\begin{rem}
Corollary \ref{RATcor:weights on geometric tropicalization} can also be proven using Cueto's multiplicity formula for geometric tropicalization \cite[Thm. 2.5]{Cue12}.
\end{rem}

\section{Applications}
\label{Asec}

In this section we show how to apply our methods to obtain classical/tropical correspondences for invariants on the moduli spaces of rational stable curves and genus $0$ logarithmic stable maps to toric varieties. We begin by showing that the tropicalizations of $\psi$-classes on $\Mnullnbar$ recover the tropical $\psi$-classes on the tropical moduli space $\Mnullntrop$. These have been defined in \cite{Mik07} as tropical cycles, and in \cite{MK09} and \cite{Rau08} as tropical Cartier divisors on  $\Mnullntrop$. We recall that the moduli space $\Mnullntrop$ of $n$-marked rational tropical curves consists of all (isomorphism classes of) metric trees with exactly $n$ unbounded edges with markings in $[n]\coloneqq \{1,\dotsc, n \}$ (cf.\ \cite{GKM09}). To make it accessible to the intersection-theoretic methods of \cite{AR10} it is usually considered as a tropical variety embedded in $\R^{\binom n2}/\Phi(\R^n)$. Here, $\Phi$ is the morphism sending the $i$-th unit vector $e_i\in\R^n$ to $\sum_{j\neq i} e_{\{i,j\}}$, where we identify the coordinates of $\R^{\binom n2}$ with two-element subsets of $[n]$. Up to a factor, the embedding usually used is given by 
\begin{equation*}
\Mnullntrop\rightarrow \R^{\binom n2} /\Phi(\R^n),\;\; \Gamma\mapsto -\sum_{i,\,j} \frac{\dist_{ij}(\Gamma)}2 e_{\{i,j\}},
\end{equation*}
 where $\dist_{ij}(\Gamma)$ denotes the distance of the vertices of the two unbounded edges of $\Gamma$ marked by $i$ and $j$. We refer to \cite[Section 3.1]{G14} for a justification of the factor $-1/2$. The image under this embedding has a canonical fan structure whose cones are the closures of the images of all tropical curves of fixed combinatorial type \cite[Thm.\ 4.2]{SS04}. From now on $\Mnullntrop$ will denote this fan. The dimension of a cone in $\Mnullntrop$ is equal to the number of bounded edges in the corresponding combinatorial type. In particular, its rays correspond to the combinatorial types with exactly one unbounded edge. These are determined by the markings on one of its vertices. For $I\subseteq [n]$ with $2\leq |I|  \leq n$ we denote by $v_I =v_{I^c}$ the primitive generator of the ray of $\Mnullntrop$ corresponding to $I$. It is equal to the image in $\R^{\binom n2}/\Phi(\R^n)$ of the tropical curve of the corresponding combinatorial type whose bounded edge has length one, that is
\begin{equation*}
v_I=- \frac 12 \sum_{i\in I,\, j\notin I}e_{\{i,j\}}.
\end{equation*}

The algebraic moduli space $\Mnullnbar$ parametrizes  rational stable  curves, that is trees of $\Pj^1$'s with $n$ pairwise different non-singular marked points. The irreducible curves form an open subset $\Mnulln\subseteq \Mnullnbar$, and since the boundary has simple normal crossings, this defines a toroidal embedding \cite[Prop.\ 4.7]{Thu07}. The strata, and hence the cones in $\Sigma(\Mnullnbar)$, are in natural bijection with the cones in $\Mnullntrop$: for two curves in $\Mnullnbar$ the dual graph construction, which replaces the components of a stable curve by nodes, their intersection points by edges, and marked points by unbounded edges with the respective marks, yields the same combinatorial type of tropical curves  if and only if they belong to the same stratum. The bijection between the cones of $\Sigma(X)$ and those of $\Mnullntrop$ even preserves dimension. In particular, every boundary divisor of $\Mnullnbar$ is equal to the divisor $D_I$ corresponding to some subset $I\subseteq [n]$ with $2\leq |I| \leq n$.

To further strengthen the connection between $\Sigma(\Mnullnbar)$ and $\Mnullntrop$ we consider the embedding
\begin{equation*}
\iota\colon\Mnulln \rightarrow \G_m^{\binom n2}/\G_m^n
\end{equation*}
given by the composite of the Gelfand-MacPherson correspondence $\operatorname{G}^0(2,n)/\G_m^n\cong \Mnulln$, where $\operatorname{G}^0(2,n)$  is the open subset of the Grassmannian with non-vanishing Plücker coordinates and $\G^n_m$ acts on it by dilating the coordinates, with the Plücker embedding  $\operatorname{G}^0(2,n)/\G_m^n\rightarrow \G_m^{\binom n2}/\G_m^n$ (modulo $\G_m^n$). The weak embedding of $\Sigma_\iota(\Mnullnbar)$ maps $|\Sigma(\Mnullnbar)|$ into $\R^{\binom n2}/\Phi(\R^n)$. Its image can be computed explicitly. It follows from \cite[Prop.\ 6.5.14]{TropBook} that the image of the primitive generator of the ray corresponding to $n\notin I\subseteq [n]$ is equal to 
\begin{equation*}
\sum_{\{i,j\}\subseteq I} e_{\{i,j\}} = \sum_{\{i,j\}\subseteq I} e_{\{i,j\}}-\frac 12 \Phi\bigg(\sum_{i\in I} e_i\bigg) = v_I
\end{equation*}
This shows that $\phi_{\Mnullnbar,\iota}$ maps the rays of $\Sigma_\iota(\Mnullnbar)$ isomorphically onto those of $\Mnullntrop$ in accordance with the bijection of cones from above. Since both cone complexes are strictly simplicial, and the rays corresponding to subsets $I_1,\dotsc, I_k$ of $[n]$ span a cone in $\Sigma_\iota(\Mnullnbar)$ if and only if they do so in $\Mnullntrop$, it follows that $\phi_{\Mnullntrop,\iota}$ induces an isomorphism $\Sigma_\iota(\Mnullnbar)\cong \Mnullntrop$ of weakly embedded cone complexes. Note that by identifying $\Mnulln$ with the complement of the union of the hyperplanes $\{x_i=x_j\}$ in $\G_m^{n-3}$ it is easy to see that we also have $\Sigma(\Mnullnbar)\cong \Sigma_\iota(\Mnullnbar)$.

In addition to the boundary divisors, there are $n$ more natural cohomology classes in $A^1(\Mnullnbar)$: the $\psi$-classes. The $k$-th $\psi$-class $\psi_k$ is defined as the first Chern class of the $k$-th cotangent line bundle on $\Mnullnbar$. We refer to \cite{psinotes} for details. Mimicking this definition, Mikhalkin has defined tropical $\psi$-classes on $\Mnullntrop$ as tropical cycles of codimension $1$ \cite{Mik07}. They have been described by Kerber and Markwig \cite{MK09} as tropical $\Q$-Cartier divisors in our sense. More precisely, they showed that Mikhalkin's  $k$-th $\psi$-class is the associated cycle of the tropical Cartier divisor $\psi_k^{\trop}$ which they defined to be the unique divisor satisfying
\begin{equation*}
\psi_k^{\trop}(v_I)=\frac{|I|(|I|-1)}{(n-1)(n-2)}
\end{equation*}
for all $I\subseteq [n]$ with $2\leq |I| \leq n-2$ and $k\notin I$. They also computed the top-dimensional intersections of these tropical $\psi$-classes explicitly and obtained that they equal their classical counterparts. A more conceptual proof of these equalities has been given by Katz \cite{Katz12} by considering $\Mnullnbar$ with its embedding into the toric variety $Y(\Mnullntrop)$ associated to the fan $\Mnullntrop$ and using the tropical description of the Chow cohomology ring of toric varieties. The following proposition strengthens this correspondence by showing that the tropical $\psi$-classes are in fact the tropicalizations of the classical ones. As a consequence we will also obtain a correspondence for intersections of $\psi$-classes of lower codimension. For the statement to make sense we consider  the classical $\psi$-classes as elements in $\Pic(\Mnullnbar)=A^1(\Mnullnbar)$ and the tropical ones as classes in $\ClTLT(\Mnullntrop)=\ClTLT(\Sigma(\Mnullnbar))$.

\begin{prop}
\label{Athm:Psi classes tropicalize to psi classes}
For every $1\leq k\leq n$ we have $\psi_k\in\ClTLT(\Mnullnbar)$ and
\begin{equation*}
\Trop_{\Mnullnbar}(\psi_k)=\psi_k^{\trop}.
\end{equation*}
In particular, for arbitrary natural numbers $a_1,\dotsc, a_n$ we have
\begin{equation*}
\Trop_{\Mnullnbar}\left(\prod_{k=1}^n\psi_k^{a_k}\cdot[\Mnullnbar]\right)=\prod_{k=1}^n(\psi_k^{\trop})^{a_k}\cdot [\Mnullnbartrop]\quad \text{in } A_*(\Mnullnbartrop),
\end{equation*}
where $[\Mnullnbartrop]=\Trop_{\Mnullnbar}([\Mnullnbar])$ is the  tropical cycle on $\Mnullnbartrop$ with weight $1$ on all maximal cones, and the tropical intersection product is that of Definition \ref{ITdefn:iterated intersections}.
\end{prop}
\begin{proof}

For the first part of the statement we may assume $k=1$, the other cases follow by symmetry. 
By \cite[1.5.2]{psinotes} the $\psi$-class $\psi_1$ is a sum of boundary divisors, namely
\begin{equation*}
\psi_1= \sum_{1\in I;\, 2,3\in J} D_I  \quad \text{in } \Pic(\Mnullnbar),
\end{equation*}
where the sum runs over subsets $I\subseteq [n]$ with $2\leq |I| \leq n-2$ which satisfy the given condition. Hence we have $\psi_1\in\ClTLT(\Mnullnbar)$, and its tropicalization is represented by 
\begin{equation*}
\sum_{1\in I;\, 2,3\in J} \phi_I,
\end{equation*}
where $\phi_I$ is the Cartier divisor on $\Mnullntrop$ which has value $1$ at $v_I$ and $0$ at $v_J$ for all $I\neq J\neq I^c$. By \cite[Lemma 2.24]{Rau08} this is equal to $\psi^{\trop}_1$ in $\ClTLT(\Mnullntrop)$. The last part of the statement is an immediate consequence of Proposition \ref{RATprop:intersection of cycle class with divisor commutes with tropicalization}.
\end{proof}

In \cite{MR08}, Markwig and Rau defined genus $0$ tropical descendant Gromov-Witten invariants of $\R^r$ by substituting tropical for algebraic objects in the classical definition of genus $0$ descendant Gromov-Witten invariants of toric varieties. In this definition, the tropical moduli space chosen to replace the moduli space of rational stable maps is the moduli space of labeled parametrized tropical curves in $\R^r$, which has been introduced in \cite{GKM09}. Tropical and classical invariants are known to be equal in special cases, but different in general. This reflects the fact that the moduli spaces of labeled parametrized tropical curves are related, yet not completely analogous to the moduli spaces of stable maps. To make this more precise, let us recall their definition. Let $n$ and $m$ be positive natural numbers and $\Delta$ an $m$-tuple of vectors in $\Z^r$ summing to $0$. Then the moduli space $\Mnullnlabtrop(\R^r,\Delta)$ consists of (isomorphism classes of) $(n+m)$-marked tropical rational curves together with continuous maps into $\R^r$ which have rational slopes on the edges, satisfy the balancing condition, contract the first $n$ marked unbounded edges, and have slope $\Delta_k$ on the $k$-th of the remaining $m$ unbounded edges (cf.\ \cite{GKM09} for details). The crucial difference to the moduli spaces of stable maps are the last $m$ markings. Their algebraic meaning is the expected number of points on an algebraic curve of specified degree mapping to the toric boundary. 

Let $\Sigma$ be a complete fan in $\R^r$ such that the elements of $\Delta$ are contained in the rays of $\Sigma$, and let $Y$ be its associated toric variety. Denote by $\Mnullnlab(Y,\Delta)$ (where $\Mnullnbarlab$ stands for rational stable maps) the set of (isomorphism classes of) maps from an $(n+m)$-marked  $\Pj^1$ to $Y$ such that the points  mapping to the boundary of $Y$ are precisely the last $m$ marked points. Moreover, we require that the $k$-th of these $m$ points is mapped to $\Orb(\R\gz\Delta_k)$, and the multiplicity with which it intersects the boundary is the index of $\Z\Delta_k$ in its saturation. As explained in detail in \cite[Section 5.2]{G14} every element in $\Mnullnlab(Y,\Delta)$ is uniquely determined by its underlying marked curve together with the image of the first marked point. Hence, we can identify $\Mnullnlab(Y,\Delta)$ with $\Mnulln[n+m]\times \G_m^r$ (cf.\ \cite[Prop.\ 3.3.3]{Ranga15}). There is a modular compactification $\Mnullnbarlab(Y,\Delta)$ of $\Mnullnlab(Y,\Delta)$, whose points correspond to genus $0$ logarithmic stable maps to $Y$. We refer to \cite{Chen14,AC14,GS13} for details. The embedding $\Mnullnlab(Y,\Delta)\rightarrow \Mnullnbarlab(Y,\Delta)$ is toroidal. In fact, if $\mu\colon\Mnullnbarlab(Y,\Delta)\rightarrow \Mnullnbar[n+m]$ denotes the morphism which forgets the map, and $\ev_i\colon\Mnullnbarlab(Y,\Delta)\rightarrow Y$ denotes the evaluation map at the $i$-th point, then $\mu\times\ev_1$ is a toroidal modification by \cite[Thm.\ 4.2.4]{Ranga15}. Hence, the weakly embedded cone complex $\Sigma(\Mnullnbarlab(Y,\Delta))$ is a proper subdivision of $\Mnullntrop[n+m]\times\Delta$. Since rational labeled parametrized tropical curves are also determined by their underlying marked (tropical) curve and the image of the first marked point \cite[Prop.\ 4.7]{GKM09}, the tropical moduli space $\Mnullnlabtrop(\R^r,\Delta)$ can be identified with $|\Sigma(\Mnullnbarlab(Y,\Delta))|$. For this reason we will use $\Mnullnlabtrop(\R^r,\Delta)$ to denote $\Sigma(\Mnullnbarlab(Y,\Delta))$. For a description of its cones we refer to \cite{Ranga15}. 

By construction, the tropicalization $\Trop(\mu)$ of the forgetful map is equal to the tropical forgetful map $\mu^{\trop}\colon \Mnullnlabtrop(\R^r,\Delta)\rightarrow \Mnullntrop[n+m]$. Similarly, the tropicalization $\Trop(\ev_i)$ of the $i$-th evaluation map is equal to the $i$-th tropical evaluation map $\ev_i^{\trop}$ by \cite[Prop.\ 5.0.2]{Ranga15}. In accordance with \cite{MR08} we define the $k$-th $\psi$-class $\hat\psi_k$ (resp.\ $\hat\psi_k^{\trop}$) on $\Mnullnbarlab(Y,\Delta)$ (resp.\ $\Mnullnbarlabtrop(\R^r,\Delta)$) as the pullback via $\mu$ (resp.\ $\mu^{\trop}$) of the $k$-th $\psi$-class $\psi_k$ (resp.\ $\psi_k^{\trop}$) on $\Mnullnbar[n+m]$ (resp.\ $\Mnullnbartrop[n+m]$). As a direct consequence of our previous results we obtain the following correspondence theorem.

\pagebreak

\begin{cor}
Let $a_1,\dotsc, a_n, b_1, \dotsc b_n$ be natural numbers, and for $1\leq i\leq n$ and $1\leq j\leq b_i$ let $D_{ij}\in\Pic(Y)=\ClTLT(Y)$. Then we have the equality
\begin{multline*}
\Trop_{\Mnullnbarlab(Y,\Delta) }\left(\prod_{i=1}^n\hat\psi_i^{a_i}\prod_{j=1}^{b_i}\ev_i^*(D_{ij})\cdot [\Mnullnbarlab(Y,\Delta)]\right)=\\
= \prod_{i=1}^n(\hat\psi_i^{\trop})^{a_i}\prod_{j=1}^{b_i}(\ev_i^{\trop})^*(\Trop_Y(D_{ij}))\cdot [\Mnullnbarlabtrop(\R^r,\Delta)]
\end{multline*}
of cycle classes in $A_*(\Mnullnbarlabtrop(\R^r,\Delta))$, where $[\Mnullnbarlabtrop(\R^r,\Delta)]=\Trop_{\Mnullnbarlab(Y,\Delta)}([\Mnullnbarlab(Y,\Delta)]\big)$ is the cycle on $\Mnullnbarlabtrop(\R^r,\Delta)$ with weight $1$ on all maximal cones.
\end{cor}

\begin{rem}
In case we take top-dimensional intersections with $Y=\Pj^r$, the tropical degree $\Delta$ containing each of the vectors $e_1,\dotsc,e_r,-\sum e_i$ exactly $d$ times, and all the $D_{ij}$ equal to classes of lines, we obtain the tropical descendant Gromov-Witten invariants of \cite{MR08} and \cite{Rau08} up to factor $(d!)^{r+1}$. In particular, slight variations of the recursion formulas of \cite{MR08,Rau08,Gauss15} also hold for the corresponding algebraic invariants.
\end{rem}

\begin{rem}
We defined the $\psi$-classes $\hat\psi_k$ on $\Mnullnbarlab(Y,\Delta)$ to be the analogues of the tropical $\psi$-classes on $\Mnullnlabtrop(\R^r,\Delta)$ of \cite{MR08}, that is as pullbacks of $\psi$-classes on $\Mnullnbar[n+m]$. I would be interested to know how they are related to the classes on $\Mnullnbarlab(Y,\Delta)$ obtained by taking Chern classes of cotangent line bundles.
\end{rem}

\bibliography{}

\normalsize
Andreas Gross, Fachbereich Mathematik, Technische Universität Kaiserslautern, Postfach 3049, 67653 Kaiserslautern, Germany, \href	{mailto:agross@mathematik.uni-kl.de}{\ttfamily agross@mathematik.uni-kl.de}

\end{document}